\documentclass{amsart}

\usepackage{amsmath,amssymb,amsthm,mathrsfs}
\usepackage[dvipdfmx]{graphicx}
\usepackage{here}
\usepackage{enumerate}
\usepackage{amsmath}
\usepackage{cases}
\usepackage{multicol}
\usepackage{amscd}
\usepackage{comment}
\usepackage{color}
\usepackage{url}
\usepackage
{hyperref}
\usepackage{booktabs}

\makeatletter
\@namedef{subjclassname@2020}{\textup{2020} Mathematics Subject Classification}
\makeatother

\newmuskip\pFqmuskip
\newcommand*\pFq[6][8]{%
  \begingroup 
  \pFqmuskip=#1mu\relax
  \mathcode`\,=\string"8000
  \begingroup\lccode`\~=`\,
  \lowercase{\endgroup\let~}\pFqcomma
  {}_{#2}F_{#3}{\left[\genfrac..{0pt}{}{#4}{#5};#6\right]}%
  \endgroup
}
\newcommand{\pFqcomma}{\mskip\pFqmuskip}

\newcommand{\Z}{\mathbb{Z}}
\newcommand{\Q}{\mathbb{Q}}
\newcommand{\R}{\mathbb{R}}
\newcommand{\C}{\mathbb{C}}
\newcommand{\D}{\mathscr{D}}
\newcommand{\spec}{\operatorname{Spec}}

\newcommand{\ord}{\operatorname{ord}}

\theoremstyle{plain}
    \newtheorem{thm}{Theorem}[section]
    \newtheorem{ppn}[thm]{Proposition}
    \newtheorem{lem}[thm]{Lemma}
    \newtheorem{cor}[thm]{Corollary}
\theoremstyle{definition}
    \newtheorem{dfn}[thm]{Definition}
\theoremstyle{remark}
    \newtheorem{rmk}[thm]{Remark}

    \newtheorem{conj}[thm]{Conjecture}

\numberwithin{equation}{section}

\title{Regulator of the Hesse cubic curves and hypergeometric functions}
\author{Yusuke Nemoto}
\email{y-nemoto@waseda.jp}
\address{Graduate School of Science and Engineering, Chiba University, 
Yayoicho 1-33, Inage, Chiba, 263-8522 Japan.}

\date{\today}
\subjclass[2020]{19F27, 33C20, 33C70}
\keywords{Regulators, $L$-functions, 
Hypergeometric functions}

\begin{document}

\maketitle

\begin{abstract}
We construct some integral elements in the motivic cohomology of the Hesse cubic curves and 
express their regulators in terms of generalized hypergeometric functions and Kamp{\'e} de F{\'e}riet hypergeometric functions. 
By using these hypergeometric expressions, 
we obtain numerical examples of the Bloch-Beilinson conjecture on special values of $L$-functions. 
\end{abstract}

\section{Introduction}
To an algebraic variety, or more generally to a motive $M$ over a number field, its $L$-function $L(M, s)$ is a holomorphic function on a complex right half plane. 
It is conjectured that $L(M, s)$ is continued analytically to a meromorphic function on the whole complex plane and satisfies a functional equation 
relating it with $L(M^{\vee}, 1-s)$, where $M^{\vee}$ is the dual motive of $M$. 
Beilinson \cite{Beilinson} conjectures that the behavior of $L(M, s)$ at an integer is explained by the regulator map, which is a map from the motivic cohomology to the Deligne cohomology. 
This conjecture is a vast generalization of the class number formula and 
the case of elliptic curves is originally due to Bloch \cite{Bloch}.
Relations among $L$-values, regulators and hypergeometric functions have been known for example by Otsubo \cite{Otsubo}, Rogers-Zudilin \cite{RZ} and Asakura \cite{A}. These objects are also studied in relation with Mahler measures: see for example Deninger \cite{Deninger}, Rodriguez Villegas \cite{Rod}, Rogers \cite{Rog} and Mellit \cite{Mellit}.

Let $X_t$ ($t \in \C \setminus \mu_3$) be the Hesse cubic curve over $\C$ defined by
 $$x_0^3+y_0^3+z_0^3=3tx_0y_0z_0,$$
 where $\mu_3$ is the cubic roots of unity. 
In this paper, we study the regulator map
$$r_{\D} \colon H^2_{\mathscr{M}}(X_t, \Q(2)) \to 
 H^2_{\mathscr{D}}(X_t, \Q(2))={\rm Hom}(H_1(X_t(\C), \Z), \C/\Q(2)). $$ 
  The group $\mu_3 \rtimes S_3$ acts on $X_t$, where $S_3$ is the symmetric group of degree $3$ (see Section 2). 
 We construct some elements in the motivic cohomology ($\zeta=e^{2 \pi i/3}$, $\rho=(1 \ 2 \ 3)$)
$$\xi(\zeta)_t, \xi(\rho z)_t \in H^2_{\mathscr{M}}(X_{t}, \Q(2)) \quad (z \in \mu_3)$$
from the $3$-torsion points on $X_t$ (viewed as an elliptic curve)  
using the method of Dokchitser, de Jeu and Zagier \cite{Zagier} (see Section 3.2).  
The main theorem of this paper (Theorem \ref{thm:main}) expresses the regulators of these elements in terms of generalized hypergeometric functions and Kamp{\'e} de F{\'e}riet hypergeometric functions defined by 
\begin{align*}
{_{3}F_{2}}\left[ \left. 
\begin{matrix}
a_1, a_2, a_3 \\
b_2,  b_2
\end{matrix}
\right| x
\right]
&=\sum_{n=0}^{\infty} \dfrac{(a_1)_n (a_2)_n (a_3)_n}{(b_1)_n  (b_2)_n}
 \dfrac{x^n}{(1)_n}, \\
F_{1;1;0}^{1;2;1}\left[\left.
\begin{matrix}
a  \\
c
\end{matrix}
; 
\begin{matrix}
b_1, b_2\\
d
\end{matrix}
;
\begin{matrix}
b^{\prime} \\
-
\end{matrix}
\right| x, y
\right]
&=
\sum_{m, n=0}^{\infty} \frac{(a)_{m+n}(b_1)_m (b_2)_m (b')_n}{(c)_{m+n}(d)_m} \frac{x^m}{(1)_m} \frac{y^n}{(1)_n},
\end{align*}
where $(\alpha)_n=\Gamma(\alpha+n)/\Gamma(\alpha)$ is the Pochhammer symbol.

Our main result is the following. 
Let $B(s, t)$ be the beta function and put $B_s=B(s, s)$. 
We put functions $(z \in \mu_3)$
\begin{align*}
&H_1(t)=B_{\frac13}t{_{3}F_{2}}\left[ \left. 
\begin{matrix}
\frac13,\frac13,\frac13 \\
\frac{4}{3}, \frac{2}{3}
\end{matrix}
\right| t^3
\right],
\quad 
H_2(t)=\frac12{B_{\frac23}}t^2{_{3}F_{2}}\left[ \left. 
\begin{matrix}
\frac23,\frac23,\frac23 \\
\frac{5}{3}, \frac{4}{3}
\end{matrix}
\right| t^3
\right], \\
&K_{1, z}(t)=
-\frac32
z
B_{\frac13}
t^2
F_{1;1;0}^{1;2;1}\left[\left.
\begin{matrix}
\frac23 \\
\frac53
\end{matrix}
; 
\begin{matrix}
\frac13, \frac13\\
\frac23
\end{matrix}
;
\begin{matrix}
1 \\
-
\end{matrix}
\right| t^3, t^3
\right]
+
z^2
B_{\frac{1}{3}}
 t^3
F_{1;1;0}^{1;2;1}\left[\left.
\begin{matrix}
1 \\
2
\end{matrix}
; 
\begin{matrix}
\frac13, \frac13\\
\frac23
\end{matrix}
;
\begin{matrix}
1 \\
-
\end{matrix}
\right| t^3, t^3
\right], \\
&K_{2, z}(t)=-
z
B_{\frac{2}{3}}
t^3
F_{1;1;0}^{1;2;1}\left[\left.
\begin{matrix}
1 \\
2
\end{matrix}
; 
\begin{matrix}
\frac23, \frac23\\
\frac43
\end{matrix}
;
\begin{matrix}
1 \\
-
\end{matrix}
\right| t^3, t^3
\right]
+
\frac34
z^2
B_{\frac{2}{3}}
t^4
F_{1;1;0}^{1;2;1}\left[\left.
\begin{matrix}
\frac43 \\
\frac73
\end{matrix}
; 
\begin{matrix}
\frac23, \frac23\\
\frac43
\end{matrix}
;
\begin{matrix}
1 \\
-
\end{matrix}
\right| t^3, t^3
\right], \\
\intertext{and constants}
&C_1=
3
 B_{\frac13}{_{3}F_{2}}\left[ \left. 
\begin{matrix}
\frac13,\frac13,1 \\
\frac{4}{3}, \frac{2}{3}
\end{matrix}
\right| 1
\right], 
\quad 
C_2
=
\frac32
B_{\frac23}{_{3}F_{2}}\left[ \left. 
\begin{matrix}
\frac23,\frac23,1 \\
\frac{5}{3}, \frac{4}{3}
\end{matrix}
\right| 1
\right].
\end{align*}
We will define a symplectic basis $\{B, A\}$ of $H_1(X_t(\C), \Z)$ in Section 2. 
\begin{thm}\label{thm:main}
As multivalued functions on $\C \setminus \mu_3$, we have the following equalities modulo $\Q(2)$. 
\begin{enumerate}
\item 
 \begin{align*}
& r_{\D}(\xi(\zeta)_t)(A)
=-9\zeta^2H_1\left(-\frac{t}{(1-t^3)^{\frac13}}\right)
+
9\zeta H_2\left(-\frac{t}{(1-t^3)^{\frac13}}\right), \\
& r_{\D}(\xi(\zeta)_t)(B)
=
-3\zeta(1-\zeta)\left(
H_1\left(-\frac{t}{(1-t^3)^{\frac13}}\right)
+
 H_2\left(-\frac{t}{(1-t^3)^{\frac13}}\right)
 \right).
\end{align*}
\item For each $z \in \mu_3$,
\begin{align*}
 r_{\D}(\xi(\rho z)_t)(A)
&=-(1-\zeta)\left(H_1(t)+K_{1, z^2}(t)\right)+(1-\zeta^2)\left(H_2(t)+K_{2, z}(t)\right)\\
&\quad -z(1-\zeta)C_1-z^2(1-\zeta^2)C_2, \\
 r_{\D}(\xi(\rho z)_t)(B)
&=H_1(t)+K_{1, z^2}(t)-H_2(t)-K_{2, z^2}(t)+zC_1+z^2C_2.
\end{align*}
\end{enumerate}
\end{thm}

A key point of our proof is to regard the Hesse cubic curves as a family.
 Then the regulator becomes a holomorphic multi-valued function of $t \in \C \setminus \mu_3$. 
By Matsumoto-Terasoma-Yamazaki \cite{kts}, 
the periods of $X_t$ are expressed in terms of Gauss's hypergeometric functions ${}_2F_1(t^3)$ (see Section 2). 
As a bi-product, we obtain a relation among hypergeometric functions from Riemann's relation  (Corollary \ref{HGR}).  
Then, by using Asakura's result \cite{A} (see Lemma \ref{prop:1}), 
 the derivative of the regulator is written in terms of the periods, hence of ${}_2F_1(t^3)$'s. 
 By a term-wise integration, we obtain Kamp{\'e} de F{\'e}riet hypergeometric functions, up to a constant. 
To determine the constant term, we look at the case $t=0$, when $X_0$ is the Fermat cubic.
For the Fermat curve of arbitrary degree, Otsubo \cite{Otsubo} expresses the regulator of Ross's element in terms of values ${}_3F_2(1)$ (see Theorem \ref{thm:Otsubo}).  
By comparing Ross's element with ours, we obtain the constant term, hence the theorem.

As an application of the main theorem, we verify numerically the Bloch-Beilinson conjecture.
Let $t \in \Q \setminus \{1\}$ and $X_t$ denote the Hesse cubic curve over $\Q$ and 
let $L(X_t, s)$ be the $L$-function of (the first cohomology of) $X_t$. 
Let $$r_{\D, \Q} \colon H^2_{\mathscr{M}}(X_t, \Q(2))  
\to \operatorname{Hom}(H_1(X_t(\C), \Q)^-, \R(1))$$ 
be the regulator map over $\Q$, which is induced by $r_{\D}$ (see Section 4.3). 
Here, $-$ denotes the $(-1)$-eigenspace with respect to the action of the complex conjugation $F_{\infty}$ acting on $X_t(\C)$. 
The Bloch-Beilinson conjecture predicts that the special value $L(X_t, 2)$ is described by the regulator $r_{\D, \Q}$ of elements in the integral part of the motivic cohomology $H^2_{\mathscr{M}}(X_t, \Q(2))_{\Z}$. 
The conjecture also predicts that $\dim H^2_{\mathscr{M}}(X_t, \Q(2))_{\Z} =1$, but very little is known about this. 
The sum $\frac13 \sum_{z \in \mu_3} \xi(z \rho)_t$ is defined over $\Q$ and defines an element  $\xi_{{\rm Hes}, t} \in H^2_{\mathscr{M}}(X_t, \Q(2))$. 
In fact, this element coincides with another element considered by Deninger \cite{Deninger} and Rodriguez Villegas \cite{Rod} (Proposition \ref{prop:element}).
See Section 4.5 for the relation with their results and ours. 
 We will prove that this element is integral if and only if $3t \in \Z$ (Theorem \ref{Hesse symbol}) by computing explicitly the boundary map in $K'$-theory, solving a conjecture in \cite{Rod}.

\begin{thm} \label{num}
For $3t \in \Z$, put $\gamma=A-F_{\infty}(A) \in H_1(X_t(\C), \Q)^-$ and  
$$Q_t=\frac{1}{2\pi i}r_{\D, \Q}(\xi_{{\rm Hes}, t})(\gamma)/\left(\frac1{\pi^2}L(X_t, 2)\right).$$
Then we have numerically the following table with
the digit of precision at least 15. 
\begin{table}[H]
  \begin{tabular}{c|ccccccccc}  \hline
  $3t$   & $-20$ & $-	19$ & $-18$ & $-17$ & $-16$ & $-15$ &$-14$& $-13$ &$-12$  \\ \hline
   $Q_t$  &$-\overset{}{\dfrac{8027}{256}}$ & $-\dfrac{313}{6}$ &$-\dfrac{1953}{76}$ &$-\dfrac{1235}{36}$  & $-\dfrac{4123}{96}$ & $-\dfrac{63}{2}$ & $-\dfrac{2771}{84}$ & $-\dfrac{139}{4}$ &$-\dfrac{117}{4}$   \\ [1.5ex] \hline
  \end{tabular}
  
  \vspace{3mm}
  
  \begin{tabular}{c|ccccccccccc}  \hline
  $3t$   & $-11$ & $-	10$ & $-9$ & $-8$ & $-7$ & $-6$ &$-5$& $-4$ &$-3$ &$-2$ & $-1$  \\ \hline
   $Q_t$  &$-\overset{}{\dfrac{679}{24}}$ & $-\dfrac{1027}{32}$ &$-\dfrac{73}2$ &$-\dfrac{77}{4}$  & $-\dfrac{185}{6}$ & $-\dfrac{81}{4}$ & $-19$ & $-\dfrac{91}{4}$ &$-\dfrac{27}{2}$ & $-\dfrac{35}{4}$ & $-7$   \\ [1.5ex] \hline
  \end{tabular}
  
  \vspace{3mm}
  
\hspace{-0.9cm}
   \begin{tabular}{c|ccccccccccc}  \hline
  $3t$  & $1$ & $2$ &$4$& $5$ &$6$ & $7$ & $8$ & $9$ & $10$ & $11$              \\ \hline
   $Q_t$   & $\overset{}{\dfrac{13}{2}}$ & $\dfrac{57}{4}$  &$-\dfrac{111}{8}$& $-\dfrac{49}2$ &$-\dfrac{189}{8}$ & $-\dfrac{79}{4}$  & $-\dfrac{485}{16}$ & $-\dfrac{117}{4}$ & $-\dfrac{973}{40}$ & $-\dfrac{163}{4}$ \\ [1.5ex] \hline
  \end{tabular}
  
    \vspace{3mm}
    
    \begin{tabular}{c|cccccccccc}  \hline
  $3t$   &  $12$ & $13$ & $14$ & $15$ & $16$ & $17$ & $18$ &$19$& $20$  \\ \hline
   $Q_t$  &$-\dfrac{189}{8}$ &$-\overset{}{\dfrac{1085}{36}}$ & $-\dfrac{2717}{72}$ &$-\dfrac{279}8$ &$-\dfrac{4069}{128}$  & $-\dfrac{2443}{60}$ & $-\dfrac{1935}{56}$ & $-\dfrac{427}{12}$ & $-\dfrac{7973}{192}$    \\ [1.5ex] \hline
  \end{tabular}
\end{table}
\end{thm}

The case $t=0$ is excluded from the theorem since $\xi_{{\rm Hes}, 0}$ has the trivial regulator, which follows from Corollary \ref{maincor}.   
Some values of $Q_t$ are proved rigorously by comparison with known results (see \eqref{value}). 
The case $t=2/3$ follows by comparing Corollary \ref{maincor} with Samart's result \cite[Theorem 2]{Samart}. 

We also give examples of the Bloch-Beilinson conjecture over the imaginary quadratic field $\Q(\mu_3)$ for $t=-2$ and $-1/2$ 
($X_{-2}$ has complex multiplication but $X_{-\frac12}$ does not). 
When $t=-1/2$, we construct the integral element $\xi'_{-\frac12}$ (note that $\xi_{{\rm Hes}, -\frac12}$ is not integral) defined over $\Q$ from the $\xi(\rho z)_{-\frac12}$'s, and  verify the conjecture over $\Q$  (Theorem \ref{element2}) numerically. 
The elements $\xi(\zeta)_{-2}$ and $\xi(\zeta)_{-\frac12}$ are defined over $\Q(\mu_3)$ and we will prove that these elements are integral (Theorem \ref{element3}).  
We verify the conjecture 
over $\Q(\mu_3)$ for $t=-2$ (resp. $t=-1/2$) numerically (Theorem \ref{quad}) by computing the regulator determinant of $\{\xi_{{\rm Hes}, -2}, \xi(\zeta)_{-2} \}$ (resp. $\{\xi'_{-\frac12}, \xi(\zeta)_{-\frac12} \}$) using Theorem \ref{thm:main}. 
For $t=0$, we can use $\{\xi(\rho)_0, \xi(\rho\zeta)_{0} \}$ to compute the regulator, and determine rigorously its ratio with the $L$-value by comparison with known results (see Remark \ref{regvalue}).

\section{Periods of the Hesse cubic curve}
Let $f \colon X \rightarrow \mathbb{P}^1$ be the Hesse family of cubic curves over $\C$, whose fiber $X_t=f^{-1}(t)$ over $t \in \mathbb{A}^1$ 
is defined by 
$$x_0^3+y_0^3+z_0^3=3tx_0y_0z_0.$$
This is smooth over $S:=\mathbb{A}^1-\mu_3$, where $\mu_3$ denotes the group of cubic roots of unity.
The affine equation is written as 
$$x^3+y^3+1=3txy \quad (x=x_0/z_0, y=y_0/z_0).$$
Let $S_3$ be the symmetric group of degree 3.
Identifying the letters $1, 2, 3$ with $x_0, y_0, z_0$, $S_3$ acts on $X$ over $\mathbb{P}^1$ by
permutation of variables. 
Let $\mu_3$ act on $X$ over $\mathbb{P}^1$ by 
$$\zeta[x_0:y_0:z_0]=[\zeta x_0:\zeta^2 y_0: z_0] \quad (\zeta \in \mu_3).$$
Then, the group $\mu_3 \rtimes S_3$ acts on $X$ where the action of $S_3$ on $\mu_3$ is given by 
$$\sigma(\zeta)=  \zeta^{{\rm sgn}(\sigma)} \quad ( \sigma \in S_3, \, \zeta \in \mu_3).$$

\begin{dfn}
We define differential forms on $X_t$ ($t \in S(\C)$) by
\begin{align*}
\omega_t&= \dfrac{dy}{x^2-ty}=\frac{-dx}{y^2-tx}, \\
\eta_t&=xy\omega_t=\dfrac{xydy}{x^2-t y}=-\dfrac{xydx}{y^2-t x}.
\end{align*}
Since $\omega_t$ (resp. $\eta_t$) is of the first (resp. second) kind, they define cohomology classes $[\omega_t], [\eta_t]$ in $H^1_{\rm dR}(X_t(\C))$.
\end{dfn}

\begin{ppn} \label{prop:diff}
Let $t \in S(\C)$.
\begin{enumerate}
\item The set $\{[\omega_t], [\eta_t]\}$ is a basis of $H^1_{\rm dR}(X_t(\C))$. 
\item The group $\mu_3 \rtimes S_3$ acts on $H^1_{\rm dR}(X_t(\C))$ via the character
$$\mu_3 \rtimes S_3 \to S_3 \xrightarrow{\rm sgn} \{\pm 1\},$$
where the first arrow is the projection. 
\end{enumerate}
\end{ppn}

\begin{proof}
(i)
 We have Serre's formula for the cup product
 $$[\omega_t] \cup [\eta_t]=\frac{1}{2\pi i}
 \int_{X_t(\C)}\omega_t \wedge \eta_t = \sum_{P \in X_t(\C)}{\rm Res}_{P} \left( \eta_t\int \omega_t  \right),$$
 where $\int \omega$ is a primitive function of $\omega$ on a small neighborhood of $P$ and ${\rm Res}_P$ denotes the residue at $P$, which does not depend on the choice of $\int \omega$.
Note that $\omega_t$ is holomorphic and $\eta_t$ is holomorphic except at $R_{\zeta}=[1 :  -\zeta : 0]$ $(\zeta \in \mu_3)$. If we put $u=z_0/x_0$ and $v=y_0/x_0$, then $u$ is a local parameter at $R_{\zeta}$ and we have
\begin{align*}
\int \omega_t&=\int\frac{1}{v^2} \left(1+\frac{tu}{v^2}+ \cdots \right)du
= \dfrac{u}{v^2} +o(u).
\end{align*}
Since
$$\eta_t=\dfrac{v}{u^2} \omega_t=\dfrac{1}{u^2}\cdot \frac1v \left(1+\frac{tu}{v^2}+ \cdots \right)du,$$
it follows that
 $${\rm Res}_{R_{\zeta}}\left( \eta_t \int \omega_t\right)=-1.$$
Therefore we obtain $[\omega_t] \cup [\eta_t]=-3 \neq 0.$
Since $\operatorname{dim} H^1_{\rm dR}(X_t(\C))=2$ and the cup product is non-degenerate,
 the assertion follows. 

(ii)
The group $S_3$ is generated by $\tau=(1\ 2)$ and $\rho=(1\ 2\ 3)$.
It is easy to show that 
$\tau^*(\omega_t)=-\omega_t$, $\tau^*(\eta_t)=-\eta_t$, $\rho^*(\omega_t)=\omega_t$, and 
$\zeta^*(\omega_t)=\omega_t$, $\zeta^*(\eta_t)=\eta_t$
($\zeta \in \mu_3$). 
Since 
$$d\left(\dfrac{y^2}{x}\right)=\dfrac{x^3-1}{x^3}\eta_t,$$
we have $\rho^{\ast}(\eta_t)=\frac{1}{x^3}\eta_t \equiv \eta_t$ modulo exact forms. 
Hence $\rho^{*}([\eta_t])=[\eta_t]$ and the
assertion follows. 
\end{proof}
We recall the construction of the symplectic basis $\{B, A\}$ of $H_1(X_t(\C), \Z)$ according to \cite{kts}.
First, we construct a cycle on $X_{\frac1{\sqrt[3]{2}}}$, and for general $t \in S(\C)$, 
the cycle is defined by continuation along a path from $1/{\sqrt[3]{2}}$ to $t$.
Note that the cycles depend on the choice of a path. 
When $t=1/\sqrt[3]{2}$, 
 let $\lambda \colon [0, 1] \rightarrow X_t(\C)$
 be the path from $P:=[0:-1:1]$ to $\rho^2(P)=[-1: 0:1]$ such that the $y$-coordinate of $\lambda(s)$ is $s-1$ (then this path is contained in $X_{\frac1{\sqrt[3]{2}}}(\R)$).
Then
$$B:=(1+\rho+\rho^2)(\lambda)$$
becomes a cycle.
To define the cycle $A$, we fix $\zeta=e^{2\pi i /3}$.
Let $\lambda'$ (written $\lambda_{03}$ in \cite{kts}) be a path from $P$ to $\zeta(P)=[0:-\zeta^2:1]$ whose $y$-coordinate 
draws two line segments passing through $\zeta-1$.
Then 
$$A:=(1+\zeta+\zeta^2)(\lambda')$$
becomes a cycle.

\begin{thm}[Matsumoto-Terasoma-Yamazaki] \label{thm:1}
Put $B_s=B(s, s)$ where $B(s, t)$ is the Beta function. 
Then, as multi-valued functions on $ \C\backslash \mu_3$, 
\begin{align*}
\int_A\omega_t&=-(1-\zeta)B_{\frac13}
{_{2}F_{1}}\left[ \left. 
\begin{matrix}
\frac13,\frac13 \\
\frac{2}{3}
\end{matrix}
\right| t^3
\right]
+(1-\zeta^2)
B_{\frac23}
t{_{2}F_{1}}\left[ \left. 
\begin{matrix}
\frac23,\frac23 \\
\frac{4}{3}
\end{matrix}
\right| t^3
\right], \\
\int_B\omega_t&=
B_{\frac13}
{_{2}F_{1}}\left[ \left. 
\begin{matrix}
\frac13,\frac13 \\
\frac{2}{3}
\end{matrix}
\right| t^3
\right]
-
B_{\frac23}t
{_{2}F_{1}}\left[ \left. 
\begin{matrix}
\frac23,\frac23 \\
\frac{4}{3}
\end{matrix}
\right| t^3
\right],\\
\intertext{and}
\int_A \eta_t &= -(1-\zeta^2)
B_{\frac23}
{_{2}F_{1}}\left[ \left. 
\begin{matrix}
-\frac13,\frac23 \\
\frac{1}{3}
\end{matrix}
\right| t^3
\right]- \dfrac{1-\zeta}{2}
B_{\frac13}t^2
{_{2}F_{1}}\left[ \left. 
\begin{matrix}
\frac13,\frac43 \\
\frac{5}{3}
\end{matrix}
\right| t^3
\right], \\
\int_B \eta_t &=
B_{\frac23}
{_{2}F_{1}}\left[ \left. 
\begin{matrix}
-\frac13,\frac23 \\
\frac{1}{3}
\end{matrix}
\right| t^3
\right] + \dfrac{1}{2}
B_{\frac13}
 t^2{_{2}F_{1}}\left[ \left. 
\begin{matrix}
\frac13,\frac43 \\
\frac{5}{3}
\end{matrix}
\right| t^3
\right].
\end{align*}
\end{thm}

\begin{proof}
The first two formulas are proved in \cite[Theorem 1]{kts}, and 
the others can be proved similarly.  
\end{proof}

As an interesting corollary, we obtain a (probably new) relation among hypergeometric functions.
\begin{cor} \label{HGR}
\begin{align*}
{_{2}F_{1}}\left[ \left. 
\begin{matrix}
\frac13, \frac13 \\
\frac23
\end{matrix}
\right| x
\right]
{_{2}F_{1}}\left[ \left. 
\begin{matrix}
-\frac13, \frac23 \\
\frac13
\end{matrix}
\right| x
\right]
+ 
\frac{x}{2}
{_{2}F_{1}}\left[ \left. 
\begin{matrix}
\frac23, \frac23 \\
\frac43
\end{matrix}
\right| x
\right]
{_{2}F_{1}}\left[ \left. 
\begin{matrix}
\frac13, \frac43 \\
\frac53
\end{matrix}
\right| x
\right]
=1
\end{align*}

\end{cor}

\begin{proof}
 By the proof of Proposition \ref{prop:diff} (i) and the period relation,
\begin{align*} 
-6\pi i=\int_{X_t(\C)} \omega_t \wedge \eta_t
=\int_{B} \omega_t \int_{A}\eta_t-\int_{A} \omega_t\int_{B}\eta_t.
\end{align*}
Then the corollary follows from Theorem \ref{thm:1} and that $B_{\frac13} B_{\frac23}=2\sqrt{3} \pi$.
\end{proof}

It might be worthwhile to write down the Gauss-Manin connection. 
\begin{cor}  \label{GM}
Put $U= f^{-1}(S)$ and 
let $\nabla \colon  \mathscr{H}^1_{\rm dR} (U/S ) \rightarrow \Omega^1_{S/{\C}} \otimes_{\mathscr{O}_S} \mathscr{H}^1_{\rm dR}(U/S)$ be the Gauss-Manin connection, where $\mathscr{H}^1_{\rm dR}(U/S)$ denotes the sheaf of relative de Rham cohomology.   
Then
$$
\nabla
\begin{pmatrix}
[\omega_t]& [\eta_t]
\end{pmatrix}
=\dfrac{dt}{1-t^3}
\otimes 
\begin{pmatrix}
[\omega_t] & [\eta_t]
\end{pmatrix}
\begin{pmatrix}
t^2 & t \\
-1& -t^2
\end{pmatrix}.
$$
\end{cor}

\begin{proof}
It amounts to show 
$$
\dfrac{d}{dt}
\begin{pmatrix}
\int _{\kappa}\omega_t& \int_{\kappa}\eta_t
\end{pmatrix}
=\dfrac{1}{1-t^3}
\begin{pmatrix}
\int_{\kappa}\omega_t & \int_{\kappa} \eta_t
\end{pmatrix}
\begin{pmatrix}
t^2 & t \\
-1& -t^2
\end{pmatrix}
$$
for any $\kappa \in H_1(U/S, \C)$.
It suffices to show this for $\kappa=\gamma, \delta$ where we put
$\gamma= A+(1-\zeta^2)B$ and $\delta=A+(1-\zeta)B$.
For $\kappa=\gamma$, 
\begin{align*}
\int_{\gamma} \omega_t = (\zeta-\zeta^2)  B_{\frac23} t{_{2}F_{1}}\left[ \left. 
\begin{matrix}
\frac23, \frac23 \\
\frac43
\end{matrix}
\right| t^3
\right], \quad
\int_{\gamma} \eta_t = -(\zeta-\zeta^2)  B_{\frac23} {_{2}F_{1}}\left[ \left. 
\begin{matrix}
-\frac13, \frac23 \\
\frac13
\end{matrix}
\right| t^3
\right]
\end{align*}
by Theorem \ref{thm:1}.
By a formal computation, 
\begin{align*}
\dfrac{d}{dt}\int_{\gamma} \omega_t= (\zeta-\zeta^2)  B_{\frac23} {_{2}F_{1}}\left[ \left. 
\begin{matrix}
\frac23, \frac23 \\
\frac13
\end{matrix}
\right| t^3
\right], \quad
\dfrac{d}{dt}\int_{\gamma} \eta_t = 2(\zeta-\zeta^2)  B_{\frac23} t^2{_{2}F_{1}}\left[ \left. 
\begin{matrix}
\frac23, \frac53 \\
\frac43
\end{matrix}
\right| t^3
\right]. 
\end{align*}
Also, we have contiguity relations of hypergeometric functions (cf. \cite[1.4]{Slater})
\begin{align*}
&\left(t\dfrac{d}{dt} + 2\right)
{_{2}F_{1}}\left[ \left. 
\begin{matrix}
\frac23, \frac23 \\
\frac43
\end{matrix}
\right| t^3
\right]=
2{_{2}F_{1}}\left[ \left. 
\begin{matrix}
\frac23, \frac53 \\
\frac43
\end{matrix}
\right| t^3
\right], \\
&\left(t\dfrac{d}{dt} -1\right)
{_{2}F_{1}}\left[ \left. 
\begin{matrix}
-\frac13, \frac23 \\
\frac13
\end{matrix}
\right| t^3
\right]=
-{_{2}F_{1}}\left[ \left. 
\begin{matrix}
\frac23, \frac23 \\
\frac13
\end{matrix}
\right| t^3
\right].
\end{align*}
Hence we obtain
\begin{align*}
\dfrac{d}{dt}\int_{\gamma} \omega_t=
t \dfrac{d}{dt} \int_{\gamma} \eta_t -\int_{\gamma} \eta_t, \,\,\,
\dfrac{d}{dt}\int_{\gamma} \eta_t =
t^2 \dfrac{d}{dt} \int_{\gamma} \omega_t+t\int_{\gamma} \omega_t.
\end{align*}
The case $\kappa=\delta$ is similar
and the result follows.
\end{proof}

\section{Construction of elements in motivic cohomology}
\subsection{Motivic cohomology group}
Let $X$ be a smooth connected quasi-projective variety of dimension $d$ over a field $k$. Let $\mathscr{K}_i$ denote the Zariski sheaf on $X$ associated to the presheaf
$$U \to K_{i}(U)$$
where $K_i(U)$ is the algebraic $K$-group of $U$ defined by Quillen. 
The Zariski cohomology groups $H^j_{\rm Zar}(X, \mathscr{K}_i)$ are called the $K$-cohomology groups. 
Let $k(X)$ be the function field of $X$. For a point $x \in X$, let $k(x)$ be the residue field at $x$ and $i_{x} \colon \{x\} \rightarrow X$ be 
the inclusion.
 Let $X^{(i)}$ be the set of points on $X$ of codimension $i$. 
Then, we have the flasque resolution 
$$0 \rightarrow \mathscr{K}_i \rightarrow K_i(k(X)) \rightarrow \bigoplus_{x \in X^{(1)}} i_{x\ast}K_{i-1}(k(x)) \rightarrow \cdots \rightarrow 
\bigoplus_{x \in X^{(d)}} i_{x\ast}K_{i-d}(k(x))\rightarrow 0.$$
Therefore, we have an isomorphism 
$$H_{\rm Zar}^j(X, \mathscr{K}_i) \cong
 \dfrac{{\rm Ker}(\oplus_{x \in X^{(j)}}K_{i-j}(k(x)) \rightarrow \oplus_{x \in X^{(j+1)}}K_{i-j-1}(k(x)))}{{\rm Im}(\oplus_{x\in X^{(j-1)}}K_{i-j+1}(k(x)) \rightarrow \oplus_{x\in X^{(j)}}K_{i-j}(k(x)))}.$$
In particular, $H^0_{\rm Zar}(X, \mathscr{K}_2)=\Gamma(X, \mathscr{K}_2)$ is identified with the 
kernel of the product map of the tame symbols 
$$T=(T_x)_{x \in X^{(1)}} \colon K_2^{M}\left(k(X)\right) \rightarrow \bigoplus_{x \in X^{(1)}} k(x)^{*}, $$
whose image lies in the direct sum. 
Here, the Milnor $K$-group $K_2^{M}\left(k(X)\right)$ is the abelian group defined by 
$$K_2^{M}\left(k(X)\right)= k(X)^{\ast} \otimes_{\Z} k(X)^{\ast} /\langle f \otimes (1-f), f \neq 0, 1 \rangle.$$
We denote the class of $f \otimes g$ by the symbol $\{f, g\}$. 
Then, $T_x$ is defined by 
$$T_x(\{f, g\})= (-1)^{{\rm ord}_x(f){{\rm ord}_x(g)}}\left(\dfrac{f^{{\rm ord}_x(g)}}{g^{{\rm ord}_x(f)}}\right)(x).$$

\begin{dfn}
We define the motivic cohomology group by
$$H^2_{\mathscr{M}}(X, \Q(2)) =\Gamma(X, \mathscr{K}_2) \otimes \Q ={\rm Ker} \,T \otimes \Q.$$
\end{dfn}

\begin{rmk}
In general, motivic cohomology groups
are defined by
$$H^n_{\mathscr{M}}(X, \Q(r)):=K^{(r)}_{2r-n}(X)_{\Q}, $$
where $(r)$ denotes the Adams eigenspace of weight $r$ \cite{Soule}. 
This agrees with our definition \cite[7.5]{Soule}.
\end{rmk}

\subsection{Construction of elements in motivic cohomology}
Dokchitser, de Jeu and Zagier  \cite{Zagier} give a method of constructing elements in the motivic cohomology 
of curves.
We briefly recall their method and  
apply to the Hesse cubic curve. 

\begin{ppn}[{\cite[Proposition 4.3]{Zagier}}] \label{const: symbol}
Let $C$ be a curve over a field $k$.
Assume that $P_1, P_2, P_3 \in C(k)$ are distinct points whose pairwise differences are torsion divisors, which means that there are rational functions $f_i$ with
$$\operatorname{div} (f_i)=m_i((P_{i+1}) - (P_{i-1})), \quad i\in \Z/3\Z,$$
 where $m_i$ is the order of $(P_{i+1}) - (P_{i-1})$ in the divisor class group $\operatorname{Pic}^0(C)$. 
\begin{enumerate}
\item There exists an element $\{P_1, P_2, P_3\} \in H^2_{\mathscr{M}}(C, \Q(2))$ such that 
$$\dfrac{\operatorname{lcm}\{m_1, m_2, m_3\}}{m_i} \{P_1, P_2, P_3\}= \left\{\dfrac{f_{i+1}}{f_{i+1}(P_{i+1})}, \dfrac{f_{i-1}}{f_{i-1}(P_{i-1})}\right\}, \quad i\in \Z/3\Z.$$
\item A permutation on the points acts on $ \{P_1, P_2, P_3\}$ as multiplication by the signature. 
\end{enumerate}
\end{ppn}

We apply this to $k=\C$ and $C=X_t$ ($t \in S(\C)$). 
Let $O=[-1 : 1: 0]$ be the origin of $X_t$ regarded as an elliptic curve. 
Let $\rho=(1\ 2\ 3) \in S_3$ and $\zeta=e^{2\pi i/3}$ be as in the proof of Proposition \ref{prop:diff} (ii). 
Then the 3-torsion points of $X_t$ are the flex points:
 $$X_t[3]=\{ \rho^i\zeta^j(O) \mid i, j \in \Z/3\Z \}.$$

\begin{dfn}\label{def:1}
We define elements of $H^2_{\mathscr{M}}(X_t, \Q(2))$ by 
\begin{align*}
&\xi(\zeta)_t=\{O, \zeta(O), \zeta^2(O)\},\\
& \xi(\rho z)_t=\{O, \rho z (O), (\rho z)^2(O)\} \quad (z \in \mu_3).
\end{align*}
When $t \in \Q(\mu_3)$, these elements are defined over $\Q(\mu_3)$. When $t \in \Q$, the element $\xi(\rho)_t$ is defined over $\Q$.
Explicitly, these elements are as follows:
\begin{align*}
&\xi(\zeta)_t=\left\{-\dfrac{t+x+y}{x+\zeta^2 y+\zeta t}, -\dfrac{ \zeta^2 t+x+\zeta y}{t+x+y}\right\}, \\
&\xi(\rho z)_t=\left\{-\dfrac{t+x+y}{z+x+z^2ty}, -\dfrac{z^2+tx+zy}{t+x+y}\right\} \quad (z \in \mu_3).
\end{align*}
\end{dfn}

On these elements, the group $\mu_3 \rtimes S_3$ acts as follows.

\begin{ppn} \label{lemm:action}
The group $\mu_3 \rtimes S_3$ acts on the elements 
$ \xi(\zeta)_t$, $\xi(\rho)_t$, $\xi(\rho \zeta)_t$ and $\xi(\rho \zeta^2)_t$
via the character as in Proposition \ref{prop:diff} (ii). 
\end{ppn}

To prove the proposition, we compute their images under the Bloch map. 
 Let $E$ be an elliptic curve over a field $k$ and $\overline{k}$ be an algebraic closure of $k$. 
 Let $\Z[E(\overline{k})]$ be the group algebra of $E(\overline{k})$. 
The Bloch map is defined as 
$$\beta \colon  \overline{k}(E)^{\ast} \otimes_{\Z} \overline{k}(E)^{*} \to \Z[E(\overline{k})], \quad \{f, g\} \mapsto \sum_{i, j}m_in_j(P_i-Q_j),$$
where $\operatorname{div} (f)=\sum_im_i (P_i)$ and $\operatorname{div} (g)=\sum_jn_j (Q_j)$.
Let $I \subset \Z[E(\overline{k})]$ be the augmentation ideal. 
Then $\beta$ takes values in $I^4$.
Let $R_3^{\ast}(E) \subset \Z[E(\overline{k})]$ be the subgroup generated by the divisors 
$\beta(\{f, 1-f\})$
for all $f \in \overline{k}(E)$, $ f \neq 0, 1$. 
The elliptic Bloch group of $E$ is defined by 
$$B_3^{\ast}(E)=\left(I^4/R_3^{\ast}(E)\right)^{\operatorname{Gal}(\overline{k}/k)}.$$
We have the following theorem due to Goncharov-Levin and Brunault.

\begin{thm}[{\cite[Theorem 1.5]{Goncharov}}, {\cite[Theorem B.2]{Brunault}}] \label{thm: inj}
The composite of the maps
$$H^2_{\mathscr{M}}(E, \Q(2))/K_2(k)\otimes \Q  \xrightarrow{\beta} B_3^{\ast}(E) \otimes \Q$$
is injective. 
(Note that the group $\operatorname{Tor}(k^{*}, E(k))$ which appears in loc.  cit. is torsion.)
\end{thm}

 \begin{proof}[ Proof of Proposition \ref{lemm:action}]
 We give a proof only for $\xi(\zeta)_t$ because the other cases are proved similarly.  
 By Proposition \ref{const: symbol} (ii), we have
\begin{align*}
&\zeta^*(\xi(\zeta)_t)= \{\zeta^{-1}(O), \zeta^{-1}\zeta(O), \zeta^{-1}  \zeta^2(O)\}
= \{\zeta^2(O), O, \zeta(O)\}
=
\xi(\zeta)_t, \\
&\tau^*(\xi(\zeta)_t)=\{\tau^{-1}(O), \tau^{-1}\zeta(O), \tau^{-1}\zeta^2(O)\}
=\{O, \zeta^2(O), \zeta(O)\}
=-\xi(\zeta)_t.
\end{align*}
Hence we are left to show $\rho^*(\xi(\zeta)_t)= \xi(\zeta)_t$. 
 We write $\xi(\zeta)_t=\{f, g\}$ as in Definition \ref{def:1}. 
 Then 
 \begin{alignat*}{2}
&\operatorname{div}(f)=3(O)-3(\zeta(O)),&& \operatorname{div}(g)=3(\zeta(O))-3(\zeta^2(O)), \\
&\operatorname{div}(\rho^*(f))=3(\rho^2(O))-3(\rho^2\zeta(O)), \quad && \operatorname{div}(\rho^*(g))=3(\rho^2\zeta(O))-3(\rho^2\zeta^2(O)).
\end{alignat*}
By \cite{Kaneko}, the group law of $X_t$ is given by
\begin{align*}
&[x: y :z] + [x^{\prime}:y^{\prime}:z^{\prime}]  \\
&=[yzz^{\prime 2} -x^2x^{\prime}y^{\prime} : xyy^{\prime 2} -z^2x^{\prime}z^{\prime 2}:  xzx^{\prime 2} -y^2y^{\prime}z^{\prime}] \\
&=[xyx^{\prime 2} -z^2y^{\prime}z^{\prime} : xzz^{\prime 2} -y^2x^{\prime}y^{\prime}:  yzy^{\prime 2} -x^2x^{\prime}z^{\prime}] \\
&=[xzy^{\prime 2} -y^2x^{\prime}z^{\prime} : yzx^{\prime 2} -x^2y^{\prime}z^{\prime}:  xyz^{\prime 2} -z^2x^{\prime}y^{\prime}]. 
\end{align*}
In particular, the inverse is given by $- [x: y : z ]=[y: x : z]$. 
It follows that 
\begin{align*}
&O-\zeta(O)=\zeta^2(O), \quad O-\zeta^2(O)=\zeta(O),   \quad\zeta(O)-\zeta^2(O)=\zeta^2(O), \\
&\rho^2(O)-\rho^2\zeta(O)=\zeta^2(O),  \ \rho^2(O)-\rho^2\zeta^2(O)=\zeta(O), \ \rho^2\zeta(O)-\rho^2\zeta^2(O)=\zeta^2(O), 
\end{align*}
and hence
$$\beta(\{f, g\})=\beta(\{\rho^*(f), \rho^*(g)\})=9(2(\zeta^2(O))-(\zeta(O))-(O)).$$
Therefore, by Theorem \ref{thm: inj} for $k=\C$, we have 
$$\rho^{*}(\xi(\zeta)_t)=\xi(\zeta)_t+c \quad (c \in K_2(\C) \otimes \Q). $$ 
Since $\rho^{*}(c)=c$ and $(\rho^{*})^3$ is trivial, we conclude that $c=0$. 
\end{proof}

 \subsection{The dlog map}
 Put $U=f^{-1}(S)$ as in Corollary \ref{GM} and $Y=X \setminus U$. 
 Then $U$ is an elliptic curve over $S$, and by construction, the elements 
 $\xi(\rho)_t$, $\xi(\zeta)_t$, $\xi(\rho \zeta)_t$ and $\xi(\rho \zeta^2)_t$
 extend respectively to  $\xi(\rho)$, $\xi(\zeta)$, $\xi(\rho\zeta)$ and $\xi(\rho\zeta^2)$
 in $H^2_{\mathscr{M}}(U, \Q(2))$.
Since the bad fibers $X_1$, $X_{\zeta}$, $X_{\zeta^2}$ and $X_{\infty}$ are the standard 3-gon, $Y$ is a normal crossing divisor of $X$. 
Hence we have the $\operatorname{dlog}$ map
$$\operatorname{dlog} \colon H^2_{\mathscr{M}}(U, \Q(2))
 \rightarrow \Gamma (X, \Omega_X^2(\log Y));\,\,\, \{f, g\} \mapsto \frac{df}{f} \wedge \frac{dg}{g}$$
to the space of algebraic differential 2-forms on $X$ with logarithmic poles along $Y$. 
The image of our elements is computed as follows. 
Define a holomorphic 1-form on $X$ by 
$$\omega=\frac{dy}{x^2-ty}.$$

\begin{ppn}\label{ppn:1}
We have
\begin{align*}
&\operatorname{dlog}(\xi(\zeta))=-\frac{(1-\zeta)^3}{1-t^3}dt \wedge \omega, \\
&\operatorname{dlog}(\xi(\rho z))=\frac{(1-z^2 t)^3}{1-t^3}dt \wedge \omega \quad (z \in \mu_3).
\end{align*}
\end{ppn}
\begin{proof}
We only give a proof for $\xi(\rho)$; the others can be proved similarly. By the definition, 
\begin{align*}
\operatorname{dlog}(\xi(\rho))=&\frac{(1-y)(1+x+y)dt+(1-t)(1-y)dx+(1-t)(1+t+x)dy }{(t+x+y)(1+x+ty)} \\
& \wedge \frac{-(1-x)(1+x+y)dt-(1-t)(1+y+t)dx-(1-t)(1-x)dy}{(1+tx+y)(t+x+y)} \\
=&-\frac{(1-t)(1-y)(1+x+y)}{(1+tx+y)(1+x+ty)(t+x+y)} dt\wedge dx \\
&+\frac{(1-t)(1-x)(1+x+y)}{(1+tx+y)(1+x+ty)(t+x+y)} dt\wedge dy\\
&+\frac{(1-t)^2(2+t)}{(1+tx+y)(1+x+ty)(t+x+y)} dx\wedge dy.
\end{align*}
By using the relation
\begin{align*}
dt \wedge dx=-\frac{y^2-tx}{x^2-ty}dt \wedge dy, \quad dx \wedge dy=\frac{xy}{x^2-ty}dt \wedge dy,
\end{align*}
we obtain the proposition. 
\end{proof}

\begin{ppn}\label{lemm:image}
The images of the elements $\xi(\rho)$, $\xi(\zeta)$, $\xi(\rho \zeta)$ and $\xi(\rho \zeta^2)$
under $\operatorname{dlog}$ are linearly independent. 
\end{ppn}

\begin{proof}
For each $t \in  \mu_3 \cup \{\infty\}$, choose a singular point $Q_t$ of $X_t$ as
$$Q_{t}=[t : t :1] \  (t \in \mu_3 ), \quad Q_{\infty}=[0 : 0 : 1].$$
Let $\operatorname{Res}_t \colon \Gamma(U, \Omega_X^2(\log Y)) \to \C$ be the Poincar\'e residue at $Q_t$, and 
put $\Phi = \oplus_t \operatorname{Res}_t \circ \operatorname{dlog} \colon H^2_{\mathscr{M}}(U, \Q(2)) \to \C^{\oplus 4}$.
Then one computes 
\begin{align*}
&\Phi(\xi(\zeta))=(1, 1, 1, 0), \\
&\Phi(\xi(\rho))=(0,-1,-1,1),\quad  \Phi(\xi(\rho \zeta))=(1,0,1,1),  \quad \Phi(\xi(\rho \zeta^2))=(-1,1, 0, 1).
\end{align*}
It is easy to verify that these vectors are linearly independent. 
\end{proof}

\section{Computation of regulators}
We consider the regulator map for $t \in S(\C)$,
$$r_{\D} \colon H^2_{\mathscr{M}}(X_t, \Q(2)) \rightarrow H^2_{\mathscr{D}}(X_t, \Q(2))={\rm Hom}(H_1(X_t(\C), \Z), \C/\Q(2)),$$
where $H^2_{\mathscr{D}}(X_t, \Q(2))$ is the Deligne cohomology group (cf.\,\cite{Nekovar}) and $\Q(2)=(2\pi i)^2 \Q$. 
We will describe the images of our elements constructed in the previous section in terms of 
 some hypergeometric functions.

\subsection{Kamp{\'e} de F{\'e}riet hypergeometric functions}
In this subsection, we recall the definition of  Kamp{\'e} de F{\'e}riet hypergeometric functions which are introduced in \cite{appkampe}, see also \cite{sk}. 
\begin{dfn}
 Kamp{\'e} de F{\'e}riet hypergeometric function is defined by 
\begin{align*}
&F_{C;D;D^{\prime}}^{A;B;B^{\prime}}\left[\left.
\begin{matrix}
a_{1}, \dots, a_{A}  \\
c_{1}, \dots, c_{C}
\end{matrix}
; 
\begin{matrix}
b_{1} , \dots, b_{B} \\
d_{1}, \dots, d_{D}
\end{matrix}
;
\begin{matrix}
b_{1}^{\prime}, \dots, b_{B^{\prime}}^{\prime} \\
d_{1}^{\prime}, \dots, d_{D^{\prime}}^{\prime}
\end{matrix}
\right| x, y
\right] \\
&:= 
\sum_{m,n=0}^{\infty} 
\frac{
\prod_{i=1}^{A} (a_{i})_{m+n} \prod_{i=1}^{B} (b_{i})_{m} \prod_{i=1}^{B^{\prime}} (b_{i}^{\prime})_{n} 
}{\prod_{i=1}^{C} (c_{i})_{m+n} \prod_{i=1}^{D} (d_{i})_{m} \prod_{i=1}^{D^{\prime}} (d_{i}^{\prime})_{n}
} \frac{x^{m}y^{n}}{(1)_{m}(1)_{n}}. 
\end{align*}
\end{dfn}
In this paper, we only consider the case where $A=1$, $B=2$, $B^{\prime}=C=D=1$, $D^{\prime}=0$, and $x=y$. We remark that the double series 
$$F_{1;1;0}^{1;2;1}\left[\left.
\begin{matrix}
a  \\
c
\end{matrix}
; 
\begin{matrix}
b_1, b_2\\
d
\end{matrix}
;
\begin{matrix}
b^{\prime} \\
-
\end{matrix}
\right| x, y
\right]
$$ converges absolutely 
on $|x| \leqq 1$ and $|y| \leqq 1$ when the parameters satisfy the conditions \cite[Theorem 1]{hms}
\begin{align*}
&{\rm Re}(c+d-a-b_1-b_2)> 0, \\
&{\rm Re}(c-a-b^{\prime}) > 0, \\
&{\rm Re}(c+d-a-b_1-b_2-b^{\prime})> 0.
\end{align*}

 Similarly as the generalized hypergeometric function ${}_3F_2(x)$, the function $F_{1;1;0}^{1;2;1}(x, y)$ becomes a multi-valued function on $\C^2-\{x(1-x)y(1-y)=0\}$, being a solution of a system of partial differential equations \cite[XLVII]{appkampe}.

\subsection{Regulator formulas}
For an element $\xi \in H_{\mathscr{M}}^2(U, \Q(2))$, we write $\xi_t=\xi|_{X_t}$ ($t \in S(\C)$). 
For  any $\gamma \in H_1(X_t, \Q)$, considered as a section of the locally constant sheaf of relative homology, we put 
$$F_{\gamma}(t)=r_{\D}(\xi_t)(\gamma),$$
which defines a multivalued function on $S(\C)$.

Here we compute the regulator of the elements constructed in Section 3. 
First, we prove Theorem \ref{thm:main} except for the constant terms in (ii). 
Recall that, if $\xi=\sum\{f, g\}$, then we have Beilinson's formula (\cite[Lemma 1.3.1.b.]{Beilinson}, cf. \cite[(4.4.3)]{Ra}) 
$$F_{\gamma}(t)=\sum \int_{\gamma}\log f \frac{dg}{g}- \log g(O) \dfrac{df}{f},$$
where $O$ is the origin of a loop representing $\gamma$.
Its derivative is computed by the following lemma.

 \begin{lem}[a special case of {\cite[Proposition 3.1]{A}}]\label{prop:1}
 If $\operatorname{dlog}(\xi)=dt \wedge \varphi $ in $\Gamma(U, \Omega_U^2)$,
 then 
$$\frac{d}{dt}F_{\gamma}(t)= \int_{\gamma}\varphi. $$
 \end{lem}

\begin{proof}[Proof of Theorem \ref{thm:main}]
(i) Let $\gamma=A$. 
By Lemma \ref{prop:1}, Proposition \ref{ppn:1} and Theorem \ref{thm:1}, 
\begin{align*}
\frac{dF_A}{dt}&=-\frac{(1-\zeta)^3}{1-t^3} \int_A \omega_t
=\frac{9\zeta^2B_{\frac13}}{1-t^3}
{_{2}F_{1}}\left[ \left. 
\begin{matrix}
\frac13,\frac13 \\
\frac{2}{3}
\end{matrix}
\right| t^3
\right]
+\frac{9\zeta
B_{\frac23}
}{1-t^3}
t{_{2}F_{1}}\left[ \left. 
\begin{matrix}
\frac23,\frac23 \\
\frac{4}{3}
\end{matrix}
\right| t^3
\right].
\end{align*}
The primitive functions with constant term $0$ are  computed as 
\begin{align*}
\int \dfrac{1}{1-t^3}
{_{2}F_{1}}\left[ \left. 
\begin{matrix}
\frac13,\frac13 \\
\frac{2}{3}
\end{matrix}
\right| t^3
\right]
dt
&=
 \sum_{n, m \geq 0}
 \int \dfrac{(\frac13)_n(\frac13)_n}{(\frac23)_n} \dfrac{t^{3n+3m}}{n!}
 dt \\
 &=
  \sum_{n, m \geq 0}
 \dfrac{(\frac13)_n(\frac13)_n}{(\frac23)_n} \dfrac{1}{3n+3m+1} \dfrac{t^{3n+3m+1}}{n!} \\
 &=
  t \sum_{n, m \geq 0}
 \dfrac{(\frac13)_{n+m}}{(\frac{4}{3})_{n+m}} \dfrac{(\frac13)_n(\frac13)_n}{(\frac23)_n}  \dfrac{(1)_m}{m!n!} t^{3n+3m} \\
 &= 
t F_{1;1;0}^{1;2;1}\left[\left.
\begin{matrix}
\frac13 \\
\frac43
\end{matrix}
; 
\begin{matrix}
\frac13, \frac13\\
\frac23
\end{matrix}
;
\begin{matrix}
1 \\
-
\end{matrix}
\right| t^3, t^3
\right], 
\end{align*}

\begin{align*}
\int \dfrac{t}{1-t^3}
{_{2}F_{1}}\left[ \left. 
\begin{matrix}
\frac23,\frac23 \\
\frac{4}{3}
\end{matrix}
\right| t^3
\right]
dt
&=
 \sum_{n, m \geq 0}
 \int \dfrac{(\frac23)_n(\frac23)_n}{(\frac43)_n} \dfrac{t^{3n+3m+1}}{n!}
 dt \\
 &=
  \sum_{n, m \geq 0}
 \dfrac{(\frac23)_n(\frac23)_n}{(\frac43)_n} \dfrac{1}{3n+3m+2} \dfrac{t^{3n+3m+2}}{n!} \\
 &=
  \dfrac{t^2}2 \sum_{n, m \geq 0}
 \dfrac{(\frac23)_{n+m}}{(\frac{5}{3})_{n+m}} \dfrac{(\frac23)_n(\frac23)_n}{(\frac43)_n}  \dfrac{(1)_m}{m!n!} t^{3n+3m} \\
 &= 
\dfrac{t^2}2 F_{1;1;0}^{1;2;1}\left[\left.
\begin{matrix}
\frac23 \\
\frac53
\end{matrix}
; 
\begin{matrix}
\frac23, \frac23\\
\frac43
\end{matrix}
;
\begin{matrix}
1 \\
-
\end{matrix}
\right| t^3, t^3
\right].
\end{align*}
By the reduction formula \cite[(2.2)]{da}
$$(1-x)^aF_{1;1;0}^{1;2;1}\left[\left.
\begin{matrix}
a  \\
b
\end{matrix}
; 
\begin{matrix}
b-e, c\\
d
\end{matrix}
;
\begin{matrix}
e \\
-
\end{matrix}
\right| x, x
\right]
=
{_{3}F_{2}}\left[ \left. 
\begin{matrix}
a, b-e, d-c \\
b, d
\end{matrix}
\right| \frac{x}{x-1}
\right],
$$
we have
\begin{align*}
F_{1;1;0}^{1;2;1}\left[\left.
\begin{matrix}
\frac13 \\
\frac43
\end{matrix}
; 
\begin{matrix}
\frac13, \frac13\\
\frac23
\end{matrix}
;
\begin{matrix}
1 \\
-
\end{matrix}
\right| t^3, t^3
\right]
&=
\dfrac{1}{(1-t^3)^{\frac13}} {_{3}F_{2}}\left[ \left. 
\begin{matrix}
\frac13,\frac13,\frac13 \\
\frac{4}{3}, \frac{2}{3}
\end{matrix}
\right| \frac{t^3}{t^3-1}
\right], \\
 F_{1;1;0}^{1;2;1}\left[\left.
\begin{matrix}
\frac23 \\
\frac53
\end{matrix}
; 
\begin{matrix}
\frac23, \frac23\\
\frac43
\end{matrix}
;
\begin{matrix}
1 \\
-
\end{matrix}
\right| t^3, t^3
\right] 
&=
\dfrac{1}{(1-t^3)^{\frac23}} 
{_{3}F_{2}}\left[ \left. 
\begin{matrix}
\frac23,\frac23,\frac23 \\
\frac{5}{3}, \frac{4}{3}
\end{matrix}
\right| \frac{t^3}{t^3-1}
\right]. 
\end{align*}
It remains to show $F_A(0)=0$. 
It is known that the element 
\begin{align*}
\xi(\zeta)_0=
\left\{ -\frac{x+y}{x+ \zeta^2 y}, -\frac{x+ \zeta y}{x+y}  \right\} \in H_{\mathscr{M}}^2(X_0, \Q(2))
\end{align*}
is a pull-back of an element on $\mathbb{P}^1$, so 
has the trivial regulator \cite[proof of Theorem 3]{Ross}.
Hence the first equality of (i) is proved and the second can be proved similarly letting $\gamma=B$. 
The statement (ii) follows similarly except for the constant terms. 
Before determining the constant terms, we first recall some more results from the literature below. 
\end{proof}

To determine the constant terms, we apply a special case of Otsubo's result \cite{Otsubo}. 
Consider the Ross element
$$e=\{1+x, 1+y\} \in H^2_{\mathscr{M}}(X_0, \Q(2)).$$
Ross  \cite[Theorem 1]{Ross} proved that $r_{\D}(e)\neq 0$.
Note that Ross and Otsubo use the equation $x^3+y^3=1$, which is obviously isomorphic to our $X_0$ 
by the substitution $(x, y) \mapsto (-x, -y)$.
 The differential forms $\omega^{1,1}$ and $\omega^{2, 2}$ in \cite{Otsubo} agree our $\omega_0$ and $\eta_0$, respectively. 
\begin{thm}[{\cite[Theorem 4.14]{Otsubo}}]\label{thm:Otsubo}
For positive real numbers $\alpha, \beta$, we put
$$\widetilde{F}(\alpha, \beta)= \dfrac{\Gamma(\alpha)\Gamma(\beta)}{\Gamma(\alpha+\beta+1)}\sum_{m,n \geq 0} \dfrac{(\alpha)_m(\beta)_n}{(\alpha+\beta+1)_{m+n}}.$$
Then, 
for any $\gamma \in H_1(X_0(\C), \Q)$, 
$$r_{\D}(e) (\gamma)=-\dfrac{1}{3 }\left(B_{\frac13}^{-1}\widetilde{F}\left(\dfrac 13, \dfrac 13\right)
\int_{\gamma} \omega_0
+
B_{\frac23}^{-1}\widetilde{F}\left(\dfrac 23, \dfrac 23\right)
\int_{\gamma} \eta_0
\right)
.$$
\end{thm}

\begin{rmk}[cf.\ {\cite[4.9, 4.10]{Otsubo}}] \label{rem}
The value $\widetilde{F}(\alpha, \beta)$ is written in terms of Appell's two-variable hypergeometric function $F_3$. 
We have also 
$$\widetilde{F}(\alpha, \beta)=\dfrac1{\alpha}\dfrac{\Gamma(\alpha)\Gamma(\beta)}{\Gamma(\alpha+\beta)}{_{3}F_{2}}\left[ \left. 
\begin{matrix}
\alpha, \alpha, 1 \\
\alpha+1, \alpha+\beta
\end{matrix}
\right| 1
\right].$$
\end{rmk}
Now, 
let $g$ be an automorphism of $X_0$ defined by
$g(x, y)=(\zeta^2x, \zeta^2 y). $
Then we have $g^*\omega_0=\zeta \omega_0$, $g^*\eta_0=\zeta^2 \eta_0$.

\begin{lem} \label{Ross}
We have for $i \in \Z/3\Z$, 
$$\xi(\rho \zeta^i)_0=-(g^i)^{*}(1+\rho+\rho^{2})^*(e).$$
Moreover, 
$$r_{\D}(\xi(\rho \zeta^i)_0)=-3(g^i)^{*}r_{\D}(e).$$
\end{lem}

\begin{proof}
By applying $(g^i)^*\rho^*$ and $(g^i)^*\rho^{2*}$ to $\{1+x, 1+y\}$, we have up to torsion
\begin{align*}
(g^i)^*\rho^*(e)&=(g^i)^{*}\left\{\frac{x+y}{x}, \frac{1+x}x\right\} =\left\{\frac{x+y}{x}, \frac{1+\zeta^{2i}x}{\zeta^{2i}x}\right\} \\
&= \{x+y, 1+\zeta^{2i} x\}  - \{ x+y, x\}, \\
(g^i)^*\rho^{2*} (e)&=(g^i)^*\left\{ \frac{1+y}{y}, \frac{x+y}{y}\right\}=\left\{ \frac{1+\zeta^{2i}y}{\zeta^{2i}y}, \frac{x+y}{y}\right\} \\
&= \{1+\zeta^{2i}y, x+y\} -\{y, x+y\}.
\end{align*}
Here we used $\{f, -f\}=0$ and that $\{z, f\}$ is torsion for a root of unity $z$.
We have 
\begin{align*}
&\{x+y, x\}+\{y, x+y\}=\{x+y, x\}-\{x+y, y\} \\
&=\left\{\left(1+\frac{x}{y} \right)y, \frac{x}{y}\right\} 
=\left\{ 1+\frac{x}{y}, \frac{x}{y} \right\} + \{ y, x\}- \{y, y\}
\end{align*}
is also torsion since $x^3+y^3=-1$. Hence we obtain
\begin{align*}
&(g^i)^*(1+\rho^{\ast}+\rho^{2\ast})(e) \\
&=\{1+\zeta^{2i} x, 1+\zeta^{2i} y\}-\{1+\zeta^{2i}x, x+y\}-\{x+y, 1+\zeta^{2i}y\}\\
&=\left\{\frac{1+\zeta^{2i}x}{x+y}, \frac{1+\zeta^{2i}y}{x+y} \right\}
 =-\xi(\rho \zeta^i)_0
\end{align*}
up to torsion. 
Since $\rho^{\ast}$ acts trivially on $\operatorname{Im} (r_{\D})$ by Proposition \ref{prop:diff} (ii), the second assertion follows. 
\end{proof}

\begin{proof}[Proof of Theorem \ref{thm:main} (continued)]
By Lemma \ref{Ross}, Theorem \ref{thm:Otsubo} and Theorem \ref{thm:1}, 
we have
\begin{align*}
r_{\D} (\xi(\rho \zeta^i)_0)(A)
&= B_{\frac13}^{-1} \widetilde{F}\left(\dfrac 13, \dfrac 13\right) \int_A(g^i)^*\omega_0 +  B_{\frac23}^{-1} \widetilde{F}\left(\dfrac 23, \dfrac 23\right) \int_A(g^i)^*\eta_0\\
&=-(1-\zeta)\zeta^i\widetilde{F}\left(\dfrac 13, \dfrac 13\right) - (1-\zeta^2)\zeta^{2i} \widetilde{F}\left(\dfrac 23, \dfrac 23\right).
\end{align*}
Rewriting $\widetilde{F}$ by ${}_3F_2(1)$ using Remark \ref{rem}, we obtain the first equality of (ii). 
 The second equality follows similarly. 
 \end{proof}

\begin{rmk}
Olsson's hypergeometric function (cf. \cite{Olsson}) is defined by
\begin{align*}
F_P(a, b_1, b_2, c_1, c_2 ; x, y)=\sum_{i, j=0}^{\infty}\dfrac{(a)_{i+j}(a-c_2+1)_i(b_1)_i(b_2)_j}{(a+b_2-c_2+1)_{i+j}(c_1)_i i! j!} x^i (1-y)^j.
\end{align*}
If we put $F_P(a, b, x)=F_P\left(a, a, 1, b, 1 ; x, 1-x\right)$,
then 
\begin{align*}
F_{1;1;0}^{1;2;1}\left[\left.
\begin{matrix}
a  \\
a+1
\end{matrix}
; 
\begin{matrix}
a, a \\
b
\end{matrix}
;
\begin{matrix}
1 \\
-
\end{matrix}
\right| x, x
\right] 
= F_P(a, b, x).
\end{align*}
Therefore 
\begin{align*}
& r_{\D}(\xi(\zeta)_t)(A)=
9\zeta^2
B_{\frac13} 
tF_P\left(\frac13, \frac23; t^3\right)
+\dfrac{9}{2}\zeta
B_{\frac23}
t^2F_P\left(\frac23, \frac43; t^3\right), \\
& r_{\D}(\xi(\zeta)_t)(B)=
3(\zeta-\zeta^2) 
B_{\frac13}tF_P\left(\frac13, \frac23; t^3\right)
-\dfrac{3}{2} (\zeta-\zeta^2)B_{\frac23}t^2F_P\left(\frac23, \frac43; t^3\right).
\end{align*}
\end{rmk}

\subsection{Numerical verification of the Bloch-Beilinson conjecture}
Let $E$ be an elliptic curve over $\Q$ and put $E_{\C}=E \times_{\operatorname{Spec}\Q} \operatorname{Spec}\C$. 
Let 
\begin{align*}r_{\D, \C} \colon H^2_{\mathscr{M}}(E_{\C}, \Q(2)) \to H^2_{\mathscr{D}}(E_{\C}, \R(2))& = \operatorname{Hom}(H_1(E(\C), \Q), \C/\R(2)) \\
&= \operatorname{Hom}(H_1(E(\C), \Q), \R(1))
\end{align*}
be the regulator map to the Deligne cohomology with $\R$-coefficients (cf. \cite{Nekovar}). 
Here $\R(n)=(2\pi i)^n \R$ and we used the identification $\C/\R(2)=\R(1)$ taking the imaginary part.  
It induces the regulator map for $E$
$$r_{\D, \Q} \colon H^2_{\mathscr{M}}(E, \Q(2)) \to H^2_{\mathscr{D}}(E_{\C}, \R(2))^+ = \operatorname{Hom}(H_1(E(\C), \Q)^-, \R(1)).$$
Here 
$\pm$ denotes the $(\pm 1)$-eigenspace with respect to the simultaneous action of the complex conjugation $F_{\infty}$ (infinite Frobenius) 
acting on $E(\C)$ and the complex conjugation $c_{\infty}$ acting on the coefficients.

 Let $\mathscr{E}$ be the proper flat regular minimal model of $E$ over $\Z$. The integral part of $H^2_{\mathscr{M}}(E, \Q(2))$ is defined by 
$$H^2_{\mathscr{M}}(E, \Q(2))_{\Z}=\operatorname{Im} (K_2(\mathscr{E}) \otimes {\Q} \to H^2_{\mathscr{M}}(E, \Q(2))).$$
Then we have the localization exact sequence
$$\cdots \to K_2(\mathscr{E})\otimes \Q \to K_2(E) \otimes \Q \xrightarrow{\partial=(\partial_p)} \bigoplus_{p: \text{prime}} K_1^{\prime}(\mathscr{E}_p)\otimes \Q \to \cdots, $$
where  
$\mathscr{E}_p$ is the special fiber of $\mathscr{E}$ at $p$ and 
$K^{\prime}(-) $ is the algebraic $K$-group of coherent sheaves. 
Therefore $\xi \in H^2_{\mathscr{M}}(E, \Q(2))$ is integral if and only if $\xi \in \operatorname{Ker} (\partial)$.

\begin{conj}[Bloch \cite{Bloch}, Beilinson \cite{Beilinson}] \label{Beilinson}
Let $L(E, s)$ be the $L$-function of $E$. 
Then there is a $\xi \in H^2_{\mathscr{M}}(E, \Q(2))_{\Z}$ such that 
$$ \cfrac{1}{2\pi i}r_{\D, \Q}(\xi)(\gamma )  \equiv \frac1{\pi^2}L(E, 2) \quad \pmod{\Q^{*}}$$
for any non-trivial $\gamma \in H_1(E(\C), \Q)^{-}$. 
\end{conj}

More generally, 
let $E$ be an elliptic curve over a number field $K$ and 
$\mathscr{E}$ be the proper flat regular minimal model of $E$ over the integer ring $\mathscr{O}_K$. 
For a finite place $v$ of $K$,  
we consider the boundary map
$$\partial_v \colon K_2(E) \otimes \Q \to K_1^{\prime}(\mathscr{E}_{v}) \otimes \Q .$$
It is known that 
\begin{align} \label{rmk1}
K_1^{\prime}(\mathscr{E}_v) \otimes  \Q \cong  
\left\{
\begin{array}{ll}
\Q & (v \colon \mbox{split\ multiplicative\ reduction})\\
0 & (\mbox{otherwise})
\end{array}
\right.
\end{align}
 (cf. \cite[Section 3]{Scholl}). 
 In the first case, $\mathscr{E}_v$ is a N\'eron $N$-gon for some $N$ and we write
$$\mathscr{E}_{v} =\bigcup_{\nu \in \Z/N\Z}C_{\nu}  \quad (\text{$C_{\nu} \cong \mathbb{P}^1$ if $N >1$})$$
where 
$C_0$ is the identity component and 
$C_0, C_1, \ldots, C_{N-1}$
are numbered consecutively.  
Let $\sum_j \{f_j, f^{\prime}_j\}$ be a symbol in $K_2(E)$. 
Let $D_j$ (resp. $D^{\prime}_{j}$) be the flat extension of $\operatorname{div}(f_j)$ (resp. $\operatorname{div}(f_j^{\prime})$) to $\mathscr{E}$ and 
put
$d_j(\nu) = \operatorname{deg}(D_j \cap C_{\nu})$ (resp. $d_j^{\prime}(\nu) = \operatorname{deg}(D^{\prime}_j \cap C_{\nu})$).

We will use the following formula of Schappacher and Scholl. 
\begin{ppn}[{\cite[Proposition 3.2, Erratum]{Scholl}}] \label{boundary}
Suppose that 
the closure of the supports of $D_j$, $D^{\prime}_j$ is contained in the smooth part of $\mathscr{E}$.
 Then we have
$$\partial_v\left(\sum_j\{f_j, f^{\prime}_j\}\right)=\pm\dfrac{1}{3N}\sum_{\mu, \nu \in \Z/N\Z} \sum_j d_j(\mu) d_j^{\prime}(\mu+\nu) {\bf B}_3\left(\left\langle \dfrac{\nu}{N} \right\rangle \right)\cdot \Phi_1^1.$$
Here, 
$0 \leq \langle x \rangle <1$ is the representative of $x \in \Q/\Z$, 
$${\bf{B}}_3(X)=X^3-\frac32 X^2+\frac12 X$$
is the third Bernoulli polynomial, and 
$\Phi_1^1$ 
is a basis of $K^{\prime}_1(\mathscr{E}_{v}) \otimes \Q$  given by the coordinate function of $C_0$.
\end{ppn}

Let us return to the case of the Hesse cubic curves. 
We change the notations and let $f \colon X \to \mathbb{P}^1$, 
$S \subset \mathbb{P}^1$, $U=f^{-1}(S)$ and $X_t=f^{-1}(t)$ for $t \in S(\Q)$  
denote the schemes defined as before, but over $\Q$. 
Our elements $\xi(\rho)_t$ and $\xi(\rho \zeta)_t + \xi(\rho \zeta^2)_t$ over $\C$ (Definition \ref{def:1}) descend to elements of $H_{\mathscr{M}}^2(X_t, \Q(2))$,
which we denote by the same letters. 
We also define the {\it Hesse symbol} by
$$\xi_{{\rm Hes}, t}=\frac13\sum_{z \in \mu_3} \xi(\rho z)_t \in H_{\mathscr{M}}^2(X_t, \Q(2)).$$
If $t=\frac1{\sqrt[3]{2}}$, then $F_{\infty}(B)=B$ by the definition. 
Since $H_1(X_t(\C), \Q)$ forms a locally constant sheaf on $U$, $F_{\infty}(B)=B$ holds in general. 
Hence a basis of $H_1(X_t(\C), \Q)^-$ is given by 
\begin{align} \label{gamma}
\gamma:= {A-F_{\infty}(A)}.  
\end{align}
As an immediate corollary of Theorem \ref{thm:main} (ii), we obtain the following. 
\begin{cor} \label{maincor}
For $t \in S(\Q)$, 
$$\frac{1}{2\pi i}r_{\D, \Q}(\xi_{{\rm Hes}, t})(\gamma)=\frac{\sqrt{3}}{2\pi}\left(H_1(t)+H_2(t)\right).$$
\end{cor}

\begin{rmk}
By Zudilin's formula \cite[Lemma 3.5]{A}, we can rewrite the formula above as 
\begin{align} \label{Zudilin}
&\frac{1}{2\pi i}r_{\D, \Q}(\xi_{{\rm Hes}, t})(\gamma) \\
&=- \frac13\left( 2\psi(1)-\psi\left(\frac13\right)-\psi\left(\frac23 \right)+3\log{t}-\frac2{9t^3}{_{4}F_{3}}\left[ \left. 
\begin{matrix}
\frac43,\frac53,1, 1 \\
2, 2, 2
\end{matrix}
\right| \frac{1}{t^3}
\right]\right), \notag
\end{align}
where $\psi(x)=\Gamma'(x)/\Gamma(x)$ is the digamma function. 
\end{rmk}
 For the integrality of the Hesse symbol, we have the following. 
 \begin{thm} \label{Hesse symbol}
For $t \in S(\Q)$, 
$\xi_{{\rm Hes}, t} \in H_{\mathscr{M}}^2(X_t, \Q(2))_{\Z}$ if and only if $3t \in \Z$. 
\end{thm}

\begin{proof}
 Let $K=\Q(\mu_3)$ and $X_{t, K}=X_t \times_{\spec \Q} \spec K$. 
 By the norm argument, the integrality $\xi_{{\rm Hes}, t}$ is equivalent to 
 the integrality of its image $\xi_{{\rm Hes}, t, K}$ in 
 $H^2_{\mathscr{M}}(X_{t, K}, \Q(2))$. 
 Write $3t=n/m$ where $n \in \Z$, $m \in \Z_{>0}$ and $(m, n)=1$. 
 Let $\mathscr{O}_K$ be the integer ring of $K$ and put 
 $$\mathscr{X}_{t, K} = \operatorname{Proj}\mathscr{O}_K[x_0, y_0, z_0]/(m(x_0^3+y_0^3+z_0^3)-nx_0y_0z_0).$$
Let $\widetilde{\mathscr{X}}_{t, K}$ be a minimal desingularization of $\mathscr{X}_{t, K}$, and 
for a prime $v$ of $K$, let 
$$\partial_v \colon K_2(X_{t, K})\otimes \Q \to K_1^{\prime}(\widetilde{\mathscr{X}}_{t, K, v})\otimes \Q$$
be the boundary map. 
By \eqref{rmk1}, it suffices to compute
$\partial_v(\xi_{{\rm Hes}, t, K})$ when 
$X_{t, K}$ has split multiplicative reduction at $v$. 
If $v \mid 3$, then $X_{t, K}$ has additive reduction.  
Multiplicative reduction occurs in the following cases:  
\begin{enumerate}
\item $v \nmid 3m$ and $n^3 \equiv 27m^3 \pmod{v}$, 
\item $v \mid m$.
\end{enumerate}
As we will see below, the reduction always splits. 
Fix a generator $\pi$ of $v$. 

(i) Put $N=\ord_v(n^3-27m^3)$.
Assume that $n \equiv 3m \pmod{v}$.
Put 
$$w_k=x_0+\zeta^k y_0+\zeta^{2k}z_0 \quad (k=0, 1, 2).$$ 
Then 
$\mathscr{X}_{t, K}$ is given locally by 
\begin{align*}
 &\mathscr{X}_{t, K, \mathfrak{p}_0}=\operatorname{Spec}\mathscr{O}_{K}\left[w_1/w_0, w_2/w_0\right]_{ \mathfrak{p}_0}/(w_1w_2/w_0^2-(unit)\pi^N),  \\
 &\mathscr{X}_{t, K, \mathfrak{p}_1}=\operatorname{Spec}\mathscr{O}_{K}\left[w_0/w_1, w_2/w_1\right]_{ \mathfrak{p}_1}/(w_0w_2/w_1^2-(unit)\pi^N),  \\
 &\mathscr{X}_{t, K, \mathfrak{p}_2}=\operatorname{Spec}\mathscr{O}_{K}\left[w_0/w_2, w_1/w_2\right]_{ \mathfrak{p}_2}/(w_0w_1/w_2^2-(unit)\pi^N),  
 \end{align*}
where $\mathfrak{p}_0=(w_1/w_0, w_2/w_0, \pi)$, 
$\mathfrak{p}_1=(w_0/w_1, w_2/w_1, \pi)$ and $\mathfrak{p}_2=(w_0/w_2, w_1/w_2, \pi)$.
A desingularization around the special fiber at $v$ is given by blowing up $\lfloor N/2 \rfloor$ times at each $\mathfrak{p}_i$, and its special fiber 
$\widetilde{\mathscr{X}}_{t, K, v}$ 
becomes a N\'eron $3N$-gon: 
$$\widetilde{\mathscr{X}}_{t, K, v}=\bigcup_{\nu \in \Z/3N\Z}C_{\nu}.$$
Here, 
the numbering is chosen so that $C_{kN}$ is the inverse image of $\{w_k=0\}$ ($k=0, 1, 2$). 
In view of Definition \ref{def:1}, put
$$f_{\zeta^i}=-\dfrac{x_0+y_0+tz_0}{x_0+\zeta^{2i}ty_0+\zeta^i z_0}, \quad f^{\prime}_{\zeta^i}=-\dfrac{tx_0+\zeta^iy_0+\zeta^{2i}z_0}{x_0+y_0+tz_0} \quad (i \in \Z/3\Z).$$
Then the assumption of Proposition \ref{boundary} is satisfied. 
The degrees of the divisors of $f_{\zeta^i}$  and $f'_{\zeta^i}$ on $C_{\nu}$ are computed as 
$$d_{\zeta^i}(\nu) = \left\{
\begin{array}{ll}
3 & (i, \nu)=(1, 0), (2, 0)\\
-3 & (i, \nu)=(1, 2N), (2, N)\\
0 & \mbox{otherwise},
\end{array}
\right.
$$
$$d^{\prime}_{\zeta^i}(\nu) = \left\{
\begin{array}{ll}
3 & (i, \nu)=(1, N), (2, 2N)\\
-3 & (i, \nu)=(1, 0), (2, 0)\\
0 & \mbox{otherwise}.
\end{array}
\right.
$$
Hence we have
\begin{align*}
\partial_v(\xi_{{\rm Hes}, t, K})&=\pm\dfrac{1}{9N}\sum_{\mu, \nu \in \Z/3N\Z} \sum_{i \in \Z/3\Z} d_{\zeta^i}(\mu) d_{\zeta^i}^{\prime}(\mu+\nu)\boldsymbol{B}_3\left(\left\langle \dfrac{\nu}{3N} \right\rangle \right)\cdot \Phi_1^1 \\
&=
\pm \frac{1}{N}\left(\boldsymbol{B}_3\left(\frac13\right)+\boldsymbol{B}_3\left(\frac23\right) \right)\cdot \Phi_1^1=0.
\end{align*}
Here, we used $\boldsymbol{B}_3(x)=-\boldsymbol{B}_3(1-x)$.
When $n \equiv 3m\zeta , 3m \zeta^2  \pmod{v}$, 
we can prove similarly that $\partial_v(\xi_{{{\rm Hes}, t, K}})=0$.

(ii) Put $N=\ord_v(m)$. Then 
$\mathscr{X}_{t, K}$ is given locally by 
\begin{align*}
 &\mathscr{X}_{t, K, \mathfrak{q}_0}=\operatorname{Spec}\mathscr{O}_{K}\left[y_0/x_0, z_0/x_0\right]_{ \mathfrak{q}_0}/(y_0z_0/x_0^2-(unit)\pi^N),  \\
 &\mathscr{X}_{t, K, \mathfrak{q}_1}=\operatorname{Spec}\mathscr{O}_{K}\left[x_0/y_0, z_0/y_0\right]_{ \mathfrak{q}_1}/(x_0z_0/y_0^2-(unit)\pi^N),  \\
 &\mathscr{X}_{t, K, \mathfrak{q}_2}=\operatorname{Spec}\mathscr{O}_{K}\left[x_0/z_0, y_0/z_0\right]_{ \mathfrak{q}_2}/(x_0y_0/z_0^2-(unit)\pi^N),  
 \end{align*}
where $\mathfrak{q}_0=(y_0/x_0, z_0/x_0, \pi)$, $\mathfrak{q}_1=(x_0/y_0, z_0/y_0, \pi)$ and $\mathfrak{q}_2=(x_0/z_0, y_0/z_0, \pi)$. 
Again 
by blowing up $\lfloor N/2 \rfloor$ times at each $\mathfrak{q}_i$, 
the special fiber 
$\widetilde{\mathscr{X}}_{t, K, v}$ becomes a N\'eron $3N$-gon;  
the numbering is chosen so that $C_0$, $C_N$ and $C_{2N}$ are respectively the inverse image of 
$\{z_0=0\}$, $\{x_0=0\}$ and  $\{y_0=0\}.$
Then the assumption of Proposition \ref{boundary} is also satisfied 
and  
we have
\begin{align*}
&d_{\zeta^i}(0)=d^{\prime}_{\zeta^i}(N)=3, \quad d_{\zeta^i}(N)=d_{\zeta^i}^{\prime}(2N)=0, \quad d_{\zeta^i}(2N)=d_{\zeta^i}^{\prime}(0)=-3,
\end{align*}
and $d_{\zeta^i}(\nu)=0$ for $\nu \not \in \{0, N, 2N\}$. 
By Proposition \ref{boundary}, we have
\begin{align*}
\partial_v(\xi_{{\rm Hes}, t, K})&=\pm\frac1{9N} \left(18\boldsymbol{B}_3\left(\frac13\right)-9\boldsymbol{B}_3\left(\frac23\right)\right) \cdot \Phi_1^1
=\pm \frac1{9N} \cdot \Phi_1^1 \neq 0.
\end{align*}
Therefore, it follows that $\partial_v(\xi_{{\rm Hes}, t, K})=0$ for any $v$  if and only if 
$m=1$.  
\end{proof}
\begin{proof}[Proof of Theorem \ref{num}]
It is known that $X_t$ is birationally equivalent to the Weierstrass model $E_t$ over $\Q$ (\cite[Proposition 1]{MJ})
$$y^2=x^3-27t(t^3+8)x+54(t^6-20t^3-8).$$ 
We compute $L(X_t, 2)=L(E_t, 2)$ by using Magma. 
On the other hand, special values of hypergeometric functions are computed by using Mathematica. 
By comparing these values, we obtain the theorem. We remark that we use \eqref{Zudilin} when $t>1$. 
\end{proof}

When $t=-1/2$, we can construct another integral element.
 \begin{thm} \label{element2}
Put 
$$\xi^{\prime}_{-\frac12}=2\xi(\rho)_{-\frac12} - \xi(\rho \zeta)_{-\frac12}-\xi(\rho \zeta^2)_{-\frac12}.$$ 
Then 
$\xi^{\prime}_{-\frac12}$ is an element of $H^2_{\mathscr{M}}(X_{-\frac12}, \Q(2))_{\Z}$,  
and 
we have numerically 
$$\frac{1}{2\pi i}r_{\D, \Q}(\xi^{\prime}_{-\frac12})(\gamma)=\dfrac{3^5}{2\pi^2} L(X_{-\frac12}, 2)$$
(the digit of precision is at least 15).
\end{thm}

\begin{proof}
First, $\xi(\rho)_{-\frac12}$ is defined over $\Q$, and $\xi(\rho \zeta)_{-\frac12}$ and $\xi(\rho \zeta^2)_{-\frac12}$ are defined over $K$ and conjugate to each other over $\Q$. 
Hence $\xi'_{-\frac12}$ is defined over $\Q$. 
When $t=-1/2$, 
$X_{t, K}$ has split multiplicative reduction only at the prime $v = (2)$. 
We can prove $\partial_2(\xi^{\prime}_{-\frac12})=0$ similarly as in the proof of Proposition \ref{Hesse symbol},
hence $\xi'_{-\frac12}$ is integral. 
By Theorem \ref{thm:main}, 
we have
\begin{align*}
&\frac{1}{2\pi i}r_{\D, \Q}(\xi^{\prime}_{-\frac12})(\gamma)\\
&=\dfrac{\sqrt{3}}{2\pi}
\left(2K_{1, 1}-K_{1, \zeta}-K_{1, \zeta^2}+2K_{2, 1}-K_{2, \zeta}-K_{2, \zeta^2}+3C_1-3C_2\right)\left(-\frac12\right).
\end{align*}
Then the numerical comparison is done as before.
\end{proof}

\subsection{Bloch-Beilinson conjecture over the quadratic field.}
Let $K=\Q(\mu_3)$.  
For $t \in S(K)$, let $X_{t, K}$ denote the Hesse cubic over $K$. 
Our elements $\xi(\zeta)_t$ and $\xi(\rho \zeta^i)_t$ (Definition \ref{def:1}) over $\C$ descends to an element in $H^2_{\mathscr{M}}(X_{t, K}, \Q(2))$, 
which we denote by the same letter.
The integral part $H^2_{\mathscr{M}}(X_{t, K}, \Q(2))_{\Z}$ is defined as before. 

\begin{thm} \label{element3}
For $t \in S(K)$, $\xi(\zeta)_{t}$ is integral if and only if $t^3=0, -8$ or $-1/8$.  
\end{thm}

\begin{proof}
When $X_{t, K}$ has complex multiplication, 
it has potentially good reduction everywhere and we have 
$$H^2_{\mathscr{M}}(X_{t, K}, \Q(2))=H^2_{\mathscr{M}}(X_{t, K}, \Q(2))_{\Z}. $$ 
Since $X_{t, K}$ has complex multiplication if and only if $t^3=0$, $-8$, the first two cases are proved. 
The remaining case is proved similarly as Theorem \ref{Hesse symbol}. 
\end{proof}

We verify the conjecture when $t^3=-8, -1/8$. 
Since $X_{-2, K}$, $X_{-2 \zeta, K}$ and $X_{-2\zeta^2, K}$ (resp. $X_{-\frac12, K}$, $X_{-\frac12 \zeta, K}$ and $X_{-\frac12 \zeta^2, K}$) are isomorphic to each other over $K$,   
we only study $X_{-2, K}$ (resp. $X_{-\frac12, K}$). 
Let $\xi_{{\rm Hes}, -2, K}  \in H^2_{\mathscr{M}}(X_{-2, K}, \Q(2))_{\Z}$ 
and 
$\xi'_{-\frac12, K}  \in H^2_{\mathscr{M}}(X_{-\frac12, K}, \Q(2))_{\Z}$ 
be the images of the elements 
$\xi_{{\rm Hes}, -2} \in H^2_{\mathscr{M}}(X_{-2}, \Q(2))_{\Z}$
and 
$\xi'_{-\frac12} \in H^2_{\mathscr{M}}(X_{-\frac12}, \Q(2))_{\Z}$ defined in Section 4.3, respectively. 
The regulator map 
\begin{align*}
r_{\D, K} \colon H^2_{\mathscr{M}}(X_{t, K}, \Q(2)) \to  \operatorname{Hom}(H_1(X_t(\C), \Q), \R(1))
\end{align*}
is the composition of the pull-back 
$H^2_{\mathscr{M}}(X_{t, K}, \Q(2)) \to H^2_{\mathscr{M}}(X_{t, \C}, \Q(2))$ (with respect to the fixed embedding of $K$ into $\C$) and $r_{\D, \C}$. 
Put the regulator determinant with respect to $\gamma$ defined in \eqref{gamma} and $B$ as  
\begin{align*}
R_{-2}& =\left| \operatorname{det} \begin{pmatrix}
\frac{1}{2 \pi i}r_{\D, K}(\xi_{{\rm Hes}, -2, K} )(\gamma) & \frac{1}{2 \pi i}r_{\D, K}(\xi(\zeta)_{-2} (\gamma) \\
\frac{1}{2 \pi i}r_{\D, K}(\xi_{{\rm Hes}, -2, K} )(B) & \frac{1}{2 \pi i}r_{\D, K}(\xi(\zeta)_{-2} )(B)
\end{pmatrix}\right|, \\
R_{-\frac12}& =\left| \operatorname{det} \begin{pmatrix}
\frac{1}{2 \pi i}r_{\D, K}(\xi'_{-\frac12, K})(\gamma) & \frac{1}{2 \pi i}r_{\D, K}(\xi(\zeta)_{-\frac12})(\gamma) \\
\frac{1}{2 \pi i}r_{\D, K}(\xi'_{-\frac12, K})(B) & \frac{1}{2 \pi i}r_{\D, K}(\xi(\zeta)_{-\frac12})(B)
\end{pmatrix}\right|. 
\end{align*}
If the Bloch-Beilinson conjecture for $X_{t, K}$ is true, we have $\dim_{\Q} H^2_{\mathscr{M}}(X_{t, K}, \Q(2))_{\Z}=[K:\Q]=2$, and if our elements are linearly independent, then 
$$R_t \equiv \frac{1}{\pi^4}L(X_{t, K}, 2 ) \pmod{\Q^*}. $$

\begin{thm} \label{quad}
We have numerically 
\begin{align*}
&R_{-2}=\frac{3^9}{(2\pi)^4}L(X_{-2, K}, 2), \\
&R_{-\frac12}=\frac{3^8}{2\pi^4}L(X_{-\frac12, K}, 2)  
\end{align*}
(the digit of the precision is at least $15$). 
In particular,   
$\xi_{{\rm Hes}, -2, K}$ and $\xi(\zeta)_{-2}$ (resp. $\xi'_{-\frac12, K}$ and $\xi(\zeta)_{-\frac12}$) 
generate a $2$-dimensional subspace of $H^2_{\mathscr{M}}(X_{-2, K}, \Q(2))_{\Z}$ (resp. $H^2_{\mathscr{M}}(X_{-\frac12, K}, \Q(2))_{\Z}$). 
\end{thm}

\begin{proof}
Since $r_{\D, K}(\xi_{{\rm Hes}, -2, K})=r_{\D, \Q}(\xi_{{\rm Hes}, -2})$ 
(resp. $r_{\D, K}(\xi'_{-\frac12, K})=r_{\D, \Q}(\xi'_{-\frac12})$) is fixed by $F_{\infty} \otimes c_{\infty}$ and $F_{\infty}(B)=B$, we have $r_{\D, K}(\xi_{{\rm Hes}, -2, K})(B)=0$ (resp.  $r_{\D, K}(\xi_{{\rm Hes}, -2, K})(B)=0$).   
Therefore we have
\begin{align*}
&R_{-2}= \left| \frac{1}{2 \pi i}r_{\D, K}(\xi_{{\rm Hes}, -2, K})(\gamma) \cdot \frac{1}{2 \pi i}r_{\D, K}(\xi(\zeta)_{-2})(B) \right|, \\
&R_{-\frac12}= \left| \frac{1}{2 \pi i}r_{\D, K}(\xi'_{-\frac12, K})(\gamma) \cdot \frac{1}{2 \pi i}r_{\D, K}(\xi(\zeta)_{-\frac12})(B) \right|.
\end{align*}
We compute the regulators by Theorem \ref{thm:main} using Mathematica, and the value of the $L$-function using Magma as 
\begin{align*}
&\frac{1}{2 \pi i}r_{\D, K}(\xi_{{\rm Hes}, -2, K})(\gamma) 
=-1.80071452138923251950118540574,
\\
&\frac{1}{2 \pi i}r_{\D, K}(\xi(\zeta)_{-2})(B)
 = -5.40214356416769755850355621723, \\
&L(X_{-2, K}, 2) =0.770263235106996761384701873629, \\
\intertext{and}
&\frac{1}{2 \pi i}r_{\D, K}(\xi'_{-\frac12, K})(\gamma) =14.4089573238768907909417772888, \\
&\frac{1}{2 \pi i}r_{\D, K}(\xi(\zeta)_{-\frac12})(B) =-2.31650091536356314247467082900, \\
&L(X_{-\frac12, K}, 2) =0.991115983384380609583674632211,
 \end{align*}
 and the theorem follows. 
\end{proof}

\begin{rmk}  \label{regvalue}
When $t=0$, 
we have $H^2_{\mathscr{M}}(X_{0, K}, \Q(2))_{\Z}=H^2_{\mathscr{M}}(X_{0, K}, \Q(2))$ since $X_0$ has complex multiplication, 
and put 
the regulator determinant as 
\begin{align*}
R_{0}& =\left| \operatorname{det} \begin{pmatrix}
\frac{1}{2 \pi i}r_{\D, K}(\xi{(\rho)_0} )(\gamma) & \frac{1}{2 \pi i}r_{\D, K}(\xi(\rho \zeta)_{0}) (\gamma) \\
\frac{1}{2 \pi i}r_{\D, K}(\xi(\rho)_0 )(B) & \frac{1}{2 \pi i}r_{\D, K}(\xi(\rho \zeta)_{0} )(B)
\end{pmatrix}\right|. 
\end{align*}
Similarly as the proof of Theorem \ref{quad}, we have $r_{\D, K}(\xi(\rho)_0)(B)=0$. By Theorem \ref{thm:main} (ii), we have
\begin{align*}
R_0 &= \left| \frac{1}{2 \pi i}r_{\D, K}(\xi(\rho)_0)(\gamma) \cdot \frac{1}{2 \pi i}r_{\D, K}(\xi(\rho \zeta)_0)(B) \right| \\
&=\frac{27}{4\pi^2} \left (
 B_{\frac13}{_{3}F_{2}}\left[ \left. 
\begin{matrix}
\frac13,\frac13,1 \\
\frac{4}{3}, \frac{2}{3}
\end{matrix}
\right| 1
\right]
-
\frac12
B_{\frac23}{_{3}F_{2}}\left[ \left. 
\begin{matrix}
\frac23,\frac23,1 \\
\frac{5}{3}, \frac{4}{3}
\end{matrix}
\right| 1
\right]\right)^2. 
\end{align*}
On the other hand, 
for a prime $v \nmid 3$ of $K$, 
let $\mathbb{F}_v$ denote the residue field at $v$, and let $\chi_v\colon \mathbb{F}_v^* \to \mu_3$ be the cubic residue character modulo $v$. 
Then the Jacobi sum is defined by
$$j_3(v)=-\sum_{\substack{{x, y \in \mathbb{F}_v^*}\\{x+y=1}}} \chi_v(x)\chi_v(y). $$
By \cite[p.494]{Weil}, $j_3$ is the Hecke character associated with $X_{0}$, hence
the $L$-function of $X_{0, K}$ is 
$$L(X_{0, K}, s)=L(j_3, s)^2.$$
Otsubo \cite[Theorem 5.2]{Otsubo2} and Rogers-Zudilin \cite[Theorem 1]{RZ} prove that 
$$L(j_3, 2)=\frac{2\pi}{27\sqrt{3}}
\left(B_{\frac13}{_{3}F_{2}}\left[ \left. 
\begin{matrix}
\frac13,\frac13,1 \\
\frac{4}{3}, \frac{2}{3}
\end{matrix}
\right| 1
\right]
-
\frac12
B_{\frac23}{_{3}F_{2}}\left[ \left. 
\begin{matrix}
\frac23,\frac23,1 \\
\frac{5}{3}, \frac{4}{3}
\end{matrix}
\right| 1
\right]
\right),$$
hence we have
\begin{equation*}
R_0=\dfrac{3^{10}}{(2\pi)^4}L(X_{0, K},  2).
\end{equation*}
\end{rmk}

\begin{rmk}
It is known that $X_{-2}$ and $X_0$ are isogenous over $\Q$.    
Since $L(X_{0}, 2)=L(X_{-2}, 2)$ and $L(X_{-2, K},2)=L(X_{-2}, 2)^2$, our computations suggest the following relations among ${}_3F_2$-values: 
\begin{align*}
&B_{\frac13}{_{3}F_{2}}\left[ \left. 
\begin{matrix}
\frac13,\frac13,1 \\
\frac{4}{3}, \frac{2}{3}
\end{matrix}
\right| 1
\right]
-
\frac12
B_{\frac23}{_{3}F_{2}}\left[ \left. 
\begin{matrix}
\frac23,\frac23,1 \\
\frac{5}{3}, \frac{4}{3}
\end{matrix}
\right| 1
\right] \\
& =
2B_{\frac13}{_{3}F_{2}}\left[ \left. 
\begin{matrix}
\frac13,\frac13,\frac13 \\
\frac{4}{3}, \frac{2}{3}
\end{matrix}
\right| -8
\right]
-
2
B_{\frac23}{_{3}F_{2}}\left[ \left. 
\begin{matrix}
\frac23,\frac23,\frac23 \\
\frac{5}{3}, \frac{4}{3}
\end{matrix}
\right| -8
\right] \\
& =
\frac{2}{\sqrt[3]{9}}B_{\frac13}{_{3}F_{2}}\left[ \left. 
\begin{matrix}
\frac13,\frac13,\frac13 \\
\frac{4}{3}, \frac{2}{3}
\end{matrix}
\right| \frac89
\right]
+
\frac{2}{3\sqrt[3]{3}}
B_{\frac23}{_{3}F_{2}}\left[ \left. 
\begin{matrix}
\frac23,\frac23,\frac23 \\
\frac{5}{3}, \frac{4}{3}
\end{matrix}
\right| \frac89
\right], 
\end{align*}
which does not seem to follow from known formulas. 
\end{rmk}

\subsection{Relation with Mahler measures}
The Mahler measure of a polynomial $P(x, y)$ is a real number defined by 
$$m(P)=\frac{1}{(2 \pi i)^2}\int_{T}\log|P(x, y)|\, \frac{dx}{x} \frac{dy}{y},$$
where $T=S^1 \times S^1$ is the real torus. 
Relations between Mahler measures and regulators were first found by Deninger \cite{Deninger}. 
Let $m(t)$ be the Mahler measure for 
$$P(x, y)=x^3+y^3+1-3txy.$$
The symbol $\{x, y\}$ defines an element of $H^2_{\mathscr{M}}(X_t, \Q(2))$, as one sees easily the vanishing of the tame symbols. 
Let $\mathscr{K}$ be the set of the values of the function $\frac{x^3+y^3+1}{3xy}$ on the torus $T$. 
Then $\mathscr{K}$ is a curved triangle with vertices at the cubic roots of unity, whose intersection with the real line is $[-1/3, 1]$ (\cite[Proposition 2.3]{Mellit}).
Deninger \cite[Proposition 3.3]{Deninger} 
proves that 
when $t \in S(\Q) \setminus \mathscr{K}$, 
$$\gamma':=X_t \cap \{|x|=1, |y| > 1\}$$
defines an element of $H_1(X_t(\C), \Q)^-$, and 
\begin{align} \label{RV}
m(t)=-\frac{1}{2 \pi i} r_{\D, \Q}(\{x, y\})(\gamma'). 
 \end{align}
 When $t=-1/3$, \eqref{RV} still holds by a continuity. 
Rodriguez Villegas \cite[Section 10, ($?_4$)]{Rod} suggests that 
if $3t \in \Z$, then $\{x, y\}$ is integral. 
This question is answered by Theorem \ref{Hesse symbol} since $\{x, y\}$ and $\xi_{{\rm Hes}, t}$ agree (Proposition \ref{prop:element} below). 
Rogers  \cite[Theorem 3.1 (43)]{Rog} expresses $m(t)$ 
in terms of ${}_3F_2(t^3)$-values. 
The formula of Corollary \ref{maincor} agrees with the combination of Roger's formula, \eqref{RV} and Proposition \ref{cycle} below.   

The cases when $3t=-6$, $-3$, $-2$, $-1$, $5$ are of particular interest. 
We remark that $X_{-\frac13}$ and $X_{\frac53}$ are isogenous. 
It is known that 
$X_{-2}$ (resp. $X_{-1}$, $X_{-\frac23}$, $X_{-\frac13}$) 
is isogenous to the modular curve $X_0(27)$ (resp. $X_0(54)$, $X_0(35)$, $X_0(14)$),
hence the Bloch-Beilinson conjecture is true by Beilinson \cite{Beilinson}.

By the works of 
Mellit \cite[Section 5]{Mellit}, 
Brunault \cite[Theorem 24]{Brunault2} and 
Rodriguez Villegas \cite[Section 14]{Rod},  
we have the following formulas:
\begin{align*}
m\left(-2\right)&=\frac{81}{4\pi^2}L(X_{-2}, 2) \\ 
m\left(-1\right)&=\frac{27}{2\pi^2}L(X_{-1}, 2), \\    
m\left(-\frac23\right)&=\frac{35}{4\pi^2}L(X_{-\frac23}, 2), \\  
m\left(-\frac13\right)&=\frac{7}{\pi^2}L(X_{-\frac13}, 2),  \\ 
m\left(\frac53\right)&=\frac{49}{2\pi^2}L(X_{\frac53}, 2).
\end{align*}
By these formulas and \eqref{RV} together with Proposition \ref{prop:element} and Proposition \ref{cycle} below, 
we obtain rigorously 
\begin{equation} \label{value}
Q_{-2}=-\frac{81}4, \quad Q_{-1}=-\frac{27}2, \quad Q_{-\frac23}=-\frac{35}4, \quad Q_{-\frac13}=-7, \quad Q_{\frac53}=-\frac{49}2
\end{equation}  (see Theorem \ref{num}).

\begin{ppn} \label{prop:element}
For $t \in S(\overline{\Q})$, we have $\xi_{{\rm Hes}, t}=\{x, y\}$. 
\end{ppn}

To prove the proposition, we compare their images under the modification of the Bloch map $\beta$ (see Section 3.2) as follows. 
Let $k$ be a number field and $E$ be an elliptic curve over ${k}$.  
Let $R_3(E)$ be the subgroup of $\Z[E(\overline{\Q})]$ generated by the divisors
$\beta(\{f, 1-f\})$ with $f \in \overline{\Q}(E)$, $f \neq 0, 1$, as well as the divisors $(P)+(-P)$ with $P \in E(\overline{\Q})$. 
We define the modification of the Bloch group by
$$B_3(E)=\left(\Z[E(\overline{\Q})]/R_3(E)\right)^{\operatorname{Gal}(\overline{\Q}/k)}.$$
Then Brunault {\cite[Theorem B.4]{Brunault}} proves that the canonical map
$$\varphi \colon  B^*_3(E) \otimes \Q \to B_3(E) \otimes \Q$$
is injective.

\begin{proof}[ Proof of Proposition \ref{prop:element}]
If $P=\rho^iz(O)$, then
$-P=\rho^{-i}z^{-1}(O)$ ($i \in \Z/3\Z$, $z \in \mu_3$).
Hence $\varphi(\rho^iz(O))+\varphi(\rho^{-i}z^{-1}(O))=0.$
We have
$$ \varphi \circ \beta(\{x, y\})=3\sum_{z \in \mu_3}(\varphi((z(O))) +\varphi((\rho z(O)))) - 6\sum_{z \in \mu_3} \varphi((\rho^2z(O)))=9\sum_{z \in \mu_3} (\rho z(O)).$$
On the other hand, we have
\begin{align*}
\varphi \circ \beta(\xi_{{\rm Hes},t})&=
6\sum_{z \in \mu_3}\varphi((\rho z(O)))-3 \sum_{z \in \mu_3} \varphi((\rho^2 z(O))) + 9\varphi((O)) =9\sum_{z \in \mu_3} (\rho z(O)).
\end{align*}
Since $\varphi$ and $\beta$ are injective, we conclude that $\xi_{{\rm Hes}, t}=\{x, y\}$. 
\end{proof}

\begin{rmk} \label{regRZ}
One can prove that $\xi_{{\rm Hes}, t}$ and $\{x, y\}$ have the same regulator for any $t \in S(\C)$ as follows. 
By Proposition \ref{ppn:1} and a similar computation, 
$$\operatorname{dlog}(\{x, y\})= \operatorname{dlog}(\xi_{{\rm Hes}, t})=dt \wedge \omega.$$
It remains to compare the regulators at $t=0$. Then $\{x, y\}=0$ and the regulator of $\xi_{{\rm Hes}, t}$ is trivial by Theorem  \ref{thm:main} (ii).   
\end{rmk}

\begin{ppn} \label{cycle}
Suppose that $t \in S(\C)$ and that $t$ is not contained in the interior of $\mathscr{K}$.  
Then the cycle $\gamma=A-F_{\infty}(A) \in H_1(X_t(\C), \Q)^-$ defined in \eqref{gamma} coincides with $\gamma'$. 
\end{ppn}
\begin{proof}
Since $H_1(X_t(\C), \Q)^-$ form a rank $1$ locally constant sheaf on $S(\C)$, there is a constant $c \in \Q^*$ such that $\gamma'=c\gamma$ for any $t \in S(\C) \setminus \mathscr{K}$.  
By Theorem \ref{thm:1}, we have
$$\int_{\gamma} \omega_t=\sqrt{3}i\left(B_{\frac13}
{_{2}F_{1}}\left[ \left. 
\begin{matrix}
\frac13,\frac13 \\
\frac{2}{3}
\end{matrix}
\right| t^3
\right]
+
B_{\frac23}
t{_{2}F_{1}}\left[ \left. 
\begin{matrix}
\frac23,\frac23 \\
\frac{4}{3}
\end{matrix}
\right| t^3
\right]\right).$$
By \cite[(1.8.1.11)]{Slater}, we have
\begin{align*}
B_{\frac13}{_{2}F_{1}}\left[ \left. 
\begin{matrix}
\frac13,\frac13 \\
\frac23
\end{matrix}
\right| t^3
\right]
+
B_{\frac23}
t{_{2}F_{1}}\left[ \left. 
\begin{matrix}
\frac23,\frac23 \\
\frac43
\end{matrix}
\right| t^3
\right] 
=
-\frac{B_{\frac13}B_{\frac23}}{3t}{_{2}F_{1}}\left[ \left. 
\begin{matrix}
\frac13,\frac23 \\
1
\end{matrix}
\right|\frac{1}{t^3}
\right].
\end{align*}
Note that $B_{\frac13}B_{\frac23}=2\sqrt{3} \pi$, hence we have
\begin{align*} 
t\int_{\gamma} \omega_t=-{2\pi i} {_{2}F_{1}}\left[ \left. 
\begin{matrix}
\frac13,\frac23 \\
1
\end{matrix}
\right|\frac{1}{t^3}
\right]
\xrightarrow{t \to \infty} -2 \pi i .
\end{align*}
On the other hand, 
if $t \not \in \mathscr{K}$,
we can write
$$P(x, y)=(y-y_0(x))(y-y_1(x))(y-y_2(x))$$
with $|y_0(x)| < 1$ and $|y_1(x)|, |y_2(x)| > 1$.
To see this, note that when $x=-1$, the roots are $y=0, \pm {\sqrt{-3t}}$, and $|\sqrt{-3t}|>1$ since $t \not \in \mathscr{K}$. 
Hence we have
\begin{align*} 
 \begin{split}
t\int_{\gamma'} \omega_t & = - \sum_{i=1}^2\int_{|x|=1} \frac{tdx}{y_i(x)^2-tx} 
=- \frac{1}{2} \sum_{i=1}^2 \int_{|x|=1} \frac{1}{1-\frac{1+x^3}{2txy_i(x)}} \frac{dx}{x}
\xrightarrow{t \to \infty} -2 \pi i. \label{cycle2}
\end{split}
\end{align*}
By comparison, we have $c=1$. 
When $t$ is on the boundary of $\mathscr{K}$, the statement follows by a continuity. 
\end{proof}

\section*{Acknowledgment}
I would like to thank sincerely  Noriyuki Otsubo for valuable discussions and many helpful comments on a draft version of this paper. 
I would like to thank Masanori Asakura and Ryojun Ito
 for helpful comments.
 I would also like to thank Detchat Samart for letting me know his paper \cite{Samart}. 
Finally, I would express his sincere gratitude to anonymous referee for many valuable suggestions and 
for strengthening the original version of Proposition \ref{lemm:action}. 

%
%
%
%
%

\end{document}